\documentclass[12pt]{amsart}
\usepackage{amsmath,amssymb,mathrsfs,hyperref,color}

\usepackage[nameinlink]{cleveref}
\hypersetup{colorlinks={true},linkcolor={blue},citecolor=green}
\crefformat{equation}{(#2#1#3)}

\usepackage{mathtools}
\usepackage[x11names]{xcolor}
\newtagform{blue}{\color{blue}(}{)}

\usepackage{paralist}

\usepackage{subfigure}
\usepackage{float}
\usepackage{times}
\usepackage{tikz}
\usepackage{cases}
\usetikzlibrary{intersections}
\usepackage{xcolor}
\usepackage[latin1]{inputenc}
\usepackage{bbm}
\usetikzlibrary{shadings}
\usetikzlibrary[intersections]
\usetikzlibrary{arrows,shapes}

\makeatletter
\def\setliststart#1{\setcounter{\@listctr}{#1}%
  \addtocounter{\@listctr}{-1}}
\makeatother

\makeatletter
\@addtoreset{figure}{section}
\makeatother

\usepackage{calc}
\newtheorem{theorem}{Theorem}[section]
\newtheorem{lemma}[theorem]{Lemma}
\newtheorem{proposition}[theorem]{Proposition}
\newtheorem{corollary}[theorem]{Corollary}
\newtheorem{remarks}[theorem]{Remark}
\newtheoremstyle{break}
  {\topsep}{\topsep}%
  {\itshape}{}%
  {\bfseries}{}%
  {\newline}{}%
\theoremstyle{break}
\newtheorem{example}[theorem]{Example}
\newtheorem{definition}[theorem]{Definition}
\numberwithin{equation}{section}

\newcommand{\T}{\mathbb{T}}
\newcommand{\R}{\mathbb{R}}

\newcommand{\N}{\mathbb{N}}

\newcommand{\PP}{\mathcal{P}}

\DeclareMathOperator*{\argmin}{argmin} 
\DeclareMathOperator*{\supp}{spt}
\DeclareMathOperator*{\esssup}{ess\ sup}
\DeclareMathOperator*{\ddiv}{div}
\DeclareMathOperator*{\OOO}{\overline{\Omega}}
\DeclareMathOperator*{\OO}{\Omega}

\DeclareMathOperator*{\eps}{\varepsilon}

\makeatletter
\def\moverlay{\mathpalette\mov@rlay}
\def\mov@rlay#1#2{\leavevmode\vtop{%
   \baselineskip\z@skip \lineskiplimit-\maxdimen
   \ialign{\hfil$\m@th#1##$\hfil\cr#2\crcr}}}
\newcommand{\charfusion}[3][\mathord]{
    #1{\ifx#1\mathop\vphantom{#2}\fi
        \mathpalette\mov@rlay{#2\cr#3}
      }
    \ifx#1\mathop\expandafter\displaylimits\fi}
\makeatother

\title[Weak KAM approach to first-order Mean Field Games with state constraints]{Weak KAM approach to first-order Mean Field Games with state constraints}
\author{Piermarco Cannarsa \and Wei Cheng \and Cristian Mendico \and Kaizhi Wang}

\begin{document}
\usetagform{blue}
\maketitle
\begin{abstract}
We study the asymptotic behavior of solutions to the constrained MFG system as the time horizon $T$ goes to infinity. For this purpose, we analyze first Hamilton-Jacobi equations with state constraints from the viewpoint of weak KAM theory, constructing a Mather measure for the associated variational problem. Using these results, we show that a solution to the constrained ergodic mean field games system exists and the ergodic constant is unique. Finally, we prove that any solution of the first-order constrained MFG problem on $[0,T]$ converges to the solution of the ergodic system as $T \to +\infty$.	
		\textit{2010\ Mathematics\ Subject Classification: 35D40; 35F21; 49J45; 49J53; 49L25.}
\end{abstract}

\tableofcontents

\section{Introduction}
The theory of Mean Field Games (MFG) was introduced independently by Lasry and Lions \cite{bib:LL1,bib:LL2,bib:LL3} and Huang, Malham\'e and Caines \cite{bib:HCM1,bib:HCM2} to study Nash equilibria for games with a very large number of players. Without entering technical details, let us recall that such an approach aims to describe the optimal value $u$ and distribution $m$ of players at a Nash equilibrium by a system of partial differential equations. Stochastic games are associated with a second order PDE system while deterministic games lead to the analysis of the first order system 
\begin{equation}\label{eq:omegaMFG}
\begin{cases}
\ -\partial _{t} u^{T} + H(x, Du^{T})=F(x, m^{T}(t)) & \text{in} \quad (0,T)\times\OO, 
\\ \  \partial _{t}m^{T}-\text{div}\Big(m^{T}D_{p}H(x, Du^{T}(t,x)\Big)=0  & \text{in} \quad (0,T)\times\OO,  
\\ \ m^{T}(0)=m_{0}, \quad u^{T}(T,x)=u^{f}(x), & x\in\OO.
\end{cases}
\end{equation}
where $\OO$ is an open domain in the Euclidean space or on a manifold. 
Following the above seminal works, this subject grew very fast producing an enormous literature. Here, for space reasons, we only refer to \cite{bib:NC, bib:DEV, bib:BFY, bib:CD1} and the references therein. However, most of the papers on this subject assumed the configuration space $\OO$ to be the torus $\T^{d}$ or the whole Euclidean space $\R^{d}$.

In this paper we investigate the long time behavior of the solution to \cref{eq:omegaMFG} where $\OO$ is a bounded domain of $\R^{d}$ and the state of the system is constrained in $\OOO$.

The constrained MFG system with finite horizon $T$ was analyzed in \cite{bib:CC, bib:CCC1, bib:CCC2}. In particular, Cannarsa and Capuani in \cite{bib:CC} introduced  the notion of constrained equilibria and mild solutions $(u^T,m^T)$ of the constrained MFG system \cref{eq:omegaMFG} with finite horizon on $\OOO$ and proved an existence and uniqueness result for such a system. In \cite{bib:CCC1, bib:CCC2}, Cannarsa, Capuani and Cardaliaguet studied the regularity of mild solutions of the constrained MFG system and used such results to give a precise interpretation of \cref{eq:omegaMFG}.

At this point, it is natural to raise the question of the asymptotic behavior of solutions as $T \to +\infty$. In the absence of state constraints, results describing the asymptotic behavior of solutions of the MFG system were obtained in \cite{bib:CLLP2, bib:CLLP1}, for second order systems on $\T^{d}$, and in \cite{bib:CCMW, bib:CAR}, for first order systems on $\T^{d}$ and $\R^{d}$, respectively. Recently, Cardaliaguet and Porretta studied the long time behavior of solutions for the so-called Master equation associated with a second order MFG system, see \cite{bib:CP2}. As is well known, the introduction of state constraints creates serious obstructions to most techniques which can be used in the unconstrained case. New methods and ideas become necessary. 

 In order to understand the specific features of constrained problems, it is useful to recall the main available results for constrained Hamilton-Jacobi equations. The dynamic programming approach to constrained optimal control problems has a long history going back to Soner \cite{bib:Soner}, who introduced the notion of constrained viscosity solutions as subsolutions in the interior of the domain and supersolutions up to the boundary. Several results followed the above seminal paper, for which we refer to \cite{bib:BD, bib:Cap-lio} and the references therein.

As for the asymptotic behavior of constrained viscosity solutions of
\begin{equation*}
\partial_{t} u(t,x) + H(x, Du(t,x))=0, \quad (t,x) \in [0,+\infty) \times \OOO
\end{equation*}
 we recall the paper \cite{bib:Mit} by Mitake, where the solution $u(t,x)$ is shown to converge as $t \to +\infty$ to a viscosity solution, $\bar{u}$,  of the ergodic Hamilton-Jacobi equation
 \begin{equation}\label{eq:1-1}
 H(x, Du(x))=c, \quad x \in \OOO	
 \end{equation}
for a unique constant $c \in \R$. 
In the absence of state constraints, it is well known that the constant $c$ can be characterized by using Mather measures, that is, invariant measures with respect to the Lagrangian flow which minimize the Lagrangian action, see for instance \cite{bib:FA}.
On the contrary, such an analysis is missing for constrained optimal control problems and the results in \cite{bib:Mit} are obtained by pure PDE methods without constructing a Mather measure.

On the other hand, as proved in \cite{bib:CCMW, bib:CAR}, the role of Mather measures is crucial in the analysis of the asymptotic behavior of solutions to the MFG system on $\T^{d}$ or $\R^{d}$. For instance, on $\T^{d}$ the limit behavior of $u^{T}$ is described by a solution $(\bar{c}, \bar{u}, \bar{m})$ of the ergodic MFG system
\begin{align*}
\begin{cases}
H(x, Du(x))=c+F(x,m), & \text{in}\ \T^{d}
\\
\ddiv\Big(mD_{p}H(x, Du(x)) \Big)=0, & \text{in}\ \T^{d}\\
\int_{\T^{d}}m(dx)=1 
\end{cases}
\end{align*}
where $\bar{m}$ is given by a Mather measure. Then, the fact that $\bar{u}$ is differentiable on the support of the Mather measure, allows to give a precise interpretation of the continuity equation in the above system.

Motivated by the above considerations, in this paper, we study the ergodic Hamil\-ton-Jacobi equation \cref{eq:1-1} from the point of view of weak KAM theory, aiming at constructing a Mather measure. For this purpose, we need to recover a fractional semiconcavity result for the value function of a constrained optimal control problem, which is inspired by a similar property derived in \cite{bib:CCC2}. Indeed, such a regularity is needed to prove the differentiability of a constrained viscosity solution of \cref{eq:1-1} along calibrated curves and, eventually, construct the Mather set.

With the above analysis at our disposal, we address the existence and uniqueness of solutions to the ergodic constrained MFG system
\begin{align}\label{eq:eMFG}
\begin{cases}
H(x, Du(x))=c+F(x,m), & \text{in}\ \OOO,
\\
\ddiv\Big(mD_{p}H(x, Du(x)) \Big)=0, & \text{in}\ \OOO,\\
\int_{\OOO}m(dx)=1. 
\end{cases}
\end{align}
As for existence, we construct a triple $(\bar{c}, \bar u, \bar m) \in \mathbb{R} \times C(\OOO) \times \mathcal{P}(\OOO)$ such that $\bar u$ is a constrained viscosity solution of the first equation in \cref{eq:eMFG} for $c=\bar{c}$, $D\bar u$ exists for $\bar m$-a.e. $x \in \OOO$, and $\bar m$ satisfies the second equation of system \cref{eq:eMFG} in the sense of distributions (for the precise definition see \Cref{def:ers}).  Moreover, under an extra monotonicity assumption for $F$, we show that $\bar{c}$ is the unique constant for which the system \cref{eq:eMFG} has a solution and $F(\cdot, \bar{m})$ is unique.


Then, using energy estimates for the MFG system, we prove our main result concerning the convergence of $u^{T} / T$: there exists a constant $C \geq $ such that 
\begin{equation*}
\sup_{t \in [0,T]} \Big\|\frac{u^{T}(t, \cdot)}{T} + \bar{c}\left(1-\frac{t}{T}\right) \Big\|_{\infty, \OOO} \leq \frac{C}{T^{\frac{1}{d+2}}}. 	
\end{equation*}


Even for the distribution of players $m^{T}$ we obtain an asymptotic estimate of the form:
\begin{equation*}
\frac{1}{T}\int_{0}^{T}{\big\| F(\cdot, m^{\eta}_s)- F(\cdot, \bar m) \big\|_{\infty, \OOO} ds} \leq \frac{C}{T^{\frac{1}{d+2}}}
\end{equation*}
for some constant $C \geq 0$. 

 We conclude this introduction recalling that asymptotic results for second-order MFG systems on $\T^{d}$ have been applied in \cite{bib:CLLP2, bib:CLLP1} to recover a Turnpike property with an exponential rate of convergence. A similar property, possibly at a lower rate, may then be expected for first-order MFG systems as well. We believe that the results of this paper could be  used to the purpose.

The rest of this paper is organized as follows.
In \Cref{sec:preliminaries}, we introduce the notation and some preliminaries. In \Cref{sec:HJstateconstraints}, we provide some weak KAM type results for Hamilton-Jacobi equations with state constraints. \Cref{sec:EMFG} is devoted to the existence of solutions of \cref{eq:MFG1}. We show the convergence result of \cref{eq:lab1} in \Cref{sec:convergence}.

\section{Preliminaries}
\label{sec:preliminaries}
\subsection{Notation}
We write below a list of symbols used throughout this paper.
\begin{itemize}
	\item Denote by $\mathbb{N}$ the set of positive integers, by $\mathbb{R}^d$ the $d$-dimensional real Euclidean space,  by $\langle\cdot,\cdot\rangle$ the Euclidean scalar product, by $|\cdot|$ the usual norm in $\mathbb{R}^d$, and by $B_{R}$ the open ball with center $0$ and radius $R$.
\item Let $\OO\subset \mathbb{R}^d$ be a bounded open set with $C^2$ boundary. $\OOO$ stands for its closure, $\partial \OO$ for its boundary and $\OO^c=\mathbb{R}^d\setminus \OO$ for the complement. For $x\in \partial \OO$, denote by $\nu(x)$ the outward unit normal vector to $\partial \OO$ at $x$.
\item The distance function from $\OOO$ is the function $d_{\OO}:\mathbb{R}^d\to[0,+\infty)$ defined by
$
d_{\OO}(x):=\inf_{y\in\OOO}|x-y|.
$
Define the oriented boundary distance from $\partial \OOO$ by $b_{\OOO}(x):=d_{\OO}(x)-d_{\OO^c}(x)$.
Since the boundary of $\OO$ is of class $C^2$, then $b_{\OO}(\cdot)$ is of class $C^2$ in a neighborhood of $\partial \OO$.

\item Denote by $\pi_1:\OOO\times \mathbb{R}^d\to\OOO$ the canonical projection.
	\item Let $\Lambda$ be a real $n\times n$ matrix. Define the norm of $\Lambda$ by
\[
\|\Lambda\|=\sup_{|x|=1,\ x\in\mathbb{R}^d}|\Lambda x|.
\]

\item Let $f$ be a real-valued function on $\mathbb{R}^d$. The set
\[
D^+ f(x)=\left\{p\in\mathbb{R}^d: \limsup_{y\to x}\frac{f(y)-f(x) - \langle p, y-x \rangle}{|y-x|}\leq 0\right\},
\]
is called the superdifferential of $f$ at $x$.

\item Let $A$ be a Lebesgue-measurable subset of $\mathbb{R}^{d}$. Let $1\leq p\leq \infty$.
Denote by $L^p(A)$ the space of Lebesgue-measurable functions $f$ with $\|f\|_{p,A}<\infty$, where
\begin{align*}
& \|f\|_{\infty, A}:=\esssup_{x \in A} |f(x)|,
\\& \|f\|_{p,A}:=\left(\int_{A}|f|^{p}\ dx\right)^{\frac{1}{p}}, \quad 1\leq p<\infty.
\end{align*}
Denote $\|f\|_{\infty,\mathbb{R}^d}$ by $\|f\|_{\infty}$ and $\|f\|_{p,\mathbb{R}^d}$ by $\|f\|_{p}$, for brevity.

\item $C_b(\mathbb{R}^d)$ stands for the function space of bounded uniformly  continuous functions on $\mathbb{R}^d$. $C^{2}_{b}(\mathbb{R}^{d})$ stands for the space of bounded functions on $\mathbb{R}^d$ with bounded uniformly continuous first and second derivatives.
$C^k(\mathbb{R}^{d})$ ($k\in\mathbb{N}$) stands for the function space of $k$-times continuously differentiable functions on $\mathbb{R}^d$, and $C^\infty(\mathbb{R}^{d}):=\cap_{k=0}^\infty C^k(\mathbb{R}^{d})$.
 $C_c^\infty(\OOO)$ stands for the space of functions $f\in C^\infty(\OOO)$ with $\supp(f)\subset \OO$. Let $a<b\in\mathbb{R}$.
  $AC([a,b];\mathbb{R}^d)$ denotes the space of absolutely continuous maps $[a,b]\to \mathbb{R}^d$.

  \item For $f \in C^{1}(\mathbb{R}^{d})$, the gradient of $f$ is denoted by $Df=(D_{x_{1}}f, ..., D_{x_{n}}f)$, where $D_{x_{i}}f=\frac{\partial f}{\partial x_{i}}$, $i=1,2,\cdots,d$.
Let $k$ be a nonnegative integer and let $\alpha=(\alpha_1,\cdots,\alpha_d)$ be a multiindex of order $k$, i.e., $k=|\alpha|=\alpha_1+\cdots +\alpha_d$ , where each component $\alpha_i$ is a nonnegative integer.   For $f \in C^{k}(\mathbb{R}^{d})$,
define $D^{\alpha}f:= D_{x_{1}}^{\alpha_{1}} \cdot\cdot\cdot D^{\alpha_{d}}_{x_{d}}f$.
\item Denote by $\mathcal{B}(\OOO)$ the Borel $\sigma$-algebra on $\OOO$, by $\mathcal{P}(\OOO)$ the set of Borel probability measures on $\OOO$, by $\mathcal{P}(\OOO\times\mathbb{R}^d)$ the set of Borel probability measures on $\OOO\times\mathbb{R}^d$. $\mathcal{P}(\OOO)$ and $\mathcal{P}(\OOO\times\mathbb{R}^d)$ are endowed with the weak-$\ast$ topology. One can define a metric on $\mathcal{P}(\OOO)$ by \cref{eq:2-100} below, which induces the weak-$\ast$ topology.
\end{itemize}

\subsection{Measure theory and MFG with state constraints}
Denote by $\mathcal{B}(\mathbb{R}^d)$ the  Borel $\sigma$-algebra on $\mathbb{R}^d$ and by $\mathcal{P}(\mathbb{R}^d)$ the space of Borel probability measures on $\mathbb{R}^d$.
The support of a measure $\mu \in \mathcal{P}(\mathbb{R}^n)$, denoted by $\supp(\mu)$, is the closed set defined by
\begin{equation*}
\supp (\mu) := \Big \{x \in \mathbb{R}^d: \mu(V_x)>0\ \text{for each open neighborhood $V_x$ of $x$}\Big\}.
\end{equation*}
We say that a sequence $\{\mu_k\}_{k\in\mathbb{N}}\subset \mathcal{P}(\mathbb{R}^d)$ is weakly-$*$ convergent to $\mu \in \mathcal{P}(\mathbb{R}^d)$, denoted by
$\mu_k \stackrel{w^*}{\longrightarrow}\mu$,
  if
\begin{equation*}
\lim_{n\rightarrow \infty} \int_{\mathbb{R}^d} f(x)\,d\mu_n(x)=\int_{\mathbb{R}^d} f(x) \,d\mu(x), \quad  \forall f \in C_b(\mathbb{R}^d).
\end{equation*}

For $p\in[1,+\infty)$, the Wasserstein space of order $p$ is defined as
\begin{equation*}
\mathcal{P}_p(\mathbb{R}^d):=\left\{m\in\mathcal{P}(\mathbb{R}^d): \int_{\mathbb{R}^d} |x_0-x|^p\,dm(x) <+\infty\right\},
\end{equation*}
where $x_0 \in \mathbb{R}^d$ is arbitrary.
Given any two measures $m$ and $m^{\prime}$ in $\mathcal{P}_p(\mathbb{R}^n)$,  define
\[
\Pi(m,m'):=\Big\{\lambda\in\mathcal{P}(\mathbb{R}^d \times \mathbb{R}^d): \lambda(A\times \mathbb{R}^d)=m(A),\ \lambda(\mathbb{R}^d \times A)=m'(A),\ \forall A\in \mathcal{B}(\mathbb{R}^d)\Big\}.
\]
The Wasserstein distance of order $p$ between $m$ and $m'$ is defined by
    \begin{equation*}\label{dis1}
          d_p(m,m')=\inf_{\lambda \in\Pi(m,m')}\left(\int_{\mathbb{R}^d\times \mathbb{R}^d}|x-y|^p\,d\lambda(x,y) \right)^{1/p}.
    \end{equation*}
    The distance $d_1$ is also commonly called the Kantorovich-Rubinstein distance and can be characterized by a useful duality formula (see, for instance, \cite{bib:CV})  as follows
\begin{equation}\label{eq:2-100}
d_1(m,m')=\sup\left\{\int_{\mathbb{R}^d} f(x)\,dm(x)-\int_{\mathbb{R}^d} f(x)\,dm'(x) \ |\ f:\mathbb{R}^d\rightarrow\mathbb{R} \ \ \text{is}\ 1\text{-Lipschitz}\right\},
\end{equation}
for all $m$, $m'\in\mathcal{P}_1(\mathbb{R}^d)$.

We recall some definitions and results for the constrained MFG system 
\begin{equation}\label{eq:lab1}
\begin{cases}
\ -\partial _{t} u^{T} + H(x, Du^{T})=F(x, m^{T}(t)) & \text{in} \quad (0,T)\times\OOO, \\ \  \partial _{t}m^{T}-\text{div}\Big(m^{T}V(t,x)\Big)=0  & \text{in} \quad (0,T)\times\OOO,  \\ \ m^{T}(0)=m_{0}, \quad u^{T}(T,x)=u^{f}(x), & x\in\OOO,
\end{cases}
\end{equation}
where 
\begin{align*}
V(t,x)=
\begin{cases}
	D_{p}H(x, Du^{T}(t,x)), & (t,x) \in [0,T] \times (\supp(m^{T}(t)) \cap \OO),
	\\
	D_{p}H(x, D^{\tau}u^{T}(t,x)+\lambda_{+}(t,x)\nu(x)), & (t,x) \in [0,T] \times (\supp(m^{T}(t)) \cap \partial\OO)
\end{cases}
\end{align*}
	and $\lambda_{+}$ is defined in \cite[Proposition 2.5]{bib:CCC2}.

Let 
	\begin{equation*}
	\Gamma=\{\gamma \in AC([0,T]; \R^{d}): \gamma(t) \in \OOO\ \text{for all}\ t \in [0,T]  \}.
	\end{equation*}
For any $x\in\OOO$, define
	\begin{equation*}
	\Gamma(x)=\{\gamma \in \Gamma: \gamma(0)=x  \}.
	\end{equation*}
For any $t \in [0,T]$, denote by $e_{t}: \Gamma \to \OOO$ the evaluation map, defined by 
\[
e_{t}(\gamma)=\gamma(t).
\]
For any $t \in [0,T]$ and any $\eta \in \PP(\Gamma)$, we define 
\[
m^{\eta}_{t}=e_{t} \sharp \eta \in \PP(\OOO)
\]
	where $e_{t} \sharp \eta$ stands for the image measure (or push-forward) of $\eta$ by $e_{t}$. Thus, for any $\varphi \in C(\overline{\Omega})$
	\begin{equation*}
		\int_{\overline{\Omega}}{\varphi(x)\ m^{\eta}_{t}(dx)}=\int_{\Gamma}{\varphi(\gamma(t))\ \eta(d\gamma)}. 
	\end{equation*}

For any fixed $m_{0} \in \PP(\OOO)$, denote by $\mathcal{P}_{m_{0}}$ the set of all Borel probability measures $\eta \in \PP(\Gamma)$ such that $m^{\eta}_{0}=m_{0}$.
For any $\eta \in \PP_{m_{0}}$ define the following functional
\begin{equation}\label{eq:evolutioncv}
J_{\eta}[\gamma]=\int_{0}^{T}{\Big(L(\gamma(s), \dot\gamma(s)) +F(\gamma(s), m^{\eta}_{s}) \Big)\ ds} + u^f(\gamma(T)),\quad \forall \gamma\in\Gamma.
\end{equation}
\begin{definition}[Constrained MFG equilibrium]\label{def:equili}
Let $m_{0} \in \PP(\OOO)$. We say that $\eta \in \PP_{m_{0}}(\Gamma)$ is a constrained MFG equilibrium for $m_{0}$ if
\begin{equation*}
\supp(\eta) \subset \Gamma^*_{\eta}:=\bigcup_{x \in \OOO} \Gamma^{\eta}(x),
\end{equation*}
where
\begin{equation*}
\Gamma^{\eta}(x)=\left\{\gamma^* \in \Gamma(x): J_{\eta}[\gamma^*]=\min_{\gamma\in\Gamma(x)} J_{\eta}[\gamma] \right\}.
\end{equation*}
\end{definition}

Assume that $L\in C^1(\OOO\times\mathbb{R}^d)$ is convex with respect to the second argument and satisfies: there are $C_i>0$, $i=1,2,3,4$, such that for all $(x,v)\in\OOO\times\mathbb{R}^d$, there hold
\[
|D_{v}L(x,v)| \leq C_1(1+|v|),\quad  |D_{x}L(x,v)| \leq C_2(1+|v|^{2}),\quad  C_3|v|^{2}- C_4 \leq L(x,v).
\]

Under these assumptions on $L$, assuming that $F:\OOO\times\mathcal{P}(\OOO)\to\mathbb{R}$ and $u^{f}:\OOO\to\mathbb{R}$ are continuous functions it has been proved in \cite[Theorem 3.1]{bib:CC} that there exists at least one constrained MFG equilibrium.

\begin{definition}[Mild solutions of constrained MFG system]\label{def:defmild}
We say that $(u^T,m^T)\in C([0,T]\times\OOO)\times C([0,T];\mathcal{P}(\OOO))$ is a mild solution of the constrained MFG problem in $\OOO$, if there is a constrained MFG equilibrium $\eta\in\mathcal{P}_{m_0}(\Gamma)$ such that
\begin{enumerate}
  \item [(i)] $m^T(t)=e_t\sharp \eta$ for all $t\in [0,T]$;
  \item [(ii)] $u^T$ is given by
  \begin{equation}\label{eq:MFGValuefunction}
  u^T(t,x)=\inf_{\gamma\in\Gamma,  \gamma(t)=x} \Big\{\int_t^T\big(L(\gamma(s),\dot{\gamma}(s))+F(\gamma(s),m^T(s))\big)ds+u^f(\gamma(T))\Big\},
  \end{equation}
  for all $(t,x)\in[0,T]\times\OOO$.
\end{enumerate}
	\end{definition}
	
	The existence of a mild solution $(u^{T}, m^{T})$ is of constraint MFG system on $[0,T]$ is a direct consequence of the existence of constrained MFG equilibrium.

In addition, assume that $F$ is strictly monotone, i.e., 
\[
\int_{\OOO}(F(x,m_1)-F(x,m_2))d(m_1-m_2)(x)\geq 0,
\]
for all $m_1$, $m_2\in\mathcal{P}(\OOO)$ and $\int_{\OOO}(F(x,m_1)-F(x,m_2))d(m_1-m_2)(x)=0$ if and only if $F(x,m_1)=F(x,m_2)$ for all $x\in\OOO$. Cannarsa and Capuani \cite{bib:CC} proved that if $(u^T_1,m^T_1)$, $(u^T_2,m^T_2)$ are mild solutions, then $u^T_1=u^T_2$. Moreover, they also provided examples of coupling functions $F$ for which also the distribution $m^{T}$ is unique under the monotonicity assumption.

\subsection{Weak KAM theory on Euclidean space}
In this part we recall some definitions and results in the weak KAM theory on the Euclidean space. Most of the results are due to Fathi \cite{bib:FM} and Contreras \cite{bib:GC}.

\medskip

\noindent $\bullet$ \emph{Tonelli Lagrangians and Hamiltonians}.
Let $L: \mathbb{R}^{n} \times \mathbb{R}^{n} \to \mathbb{R}$ be a $C^{2}$ Lagrangian.

\begin{definition}[Strict Tonelli Lagrangians]\label{def:def2}
$L$ is called a {\it strict Tonelli Lagrangian} if there exist positive constants $C_{i}$ ($i=1,2,3$) such that, for all $(x,v) \in \mathbb{R}^{d} \times \mathbb{R}^{d}$ there hold
\begin{itemize}
\item[(a)] $\frac{I}{C_{1}} \leq D_{vv}^{2}L(x,v) \leq C_{1} I$, where $I$ is the identity matrix;
	\item[(b)] $\|D^{2}_{vx}L(x,v)\| \leq C_{2}(1+|v|)$;
	\item[(c)] $|L(x,0)|+|D_{x}L(x,0)|+ |D_{v}L(x,0)| \leq C_{3}$.
	 \end{itemize}
\end{definition}

\begin{remarks}\label{rem:re2.1}
Let $L$ be a strict Tonelli Lagrangian.  It is easy to check that there are two positive constants $\alpha$, $\beta$ depending only on $C_{i}$ ($i=1,2,3$) in \Cref{def:def2},  such that
	\begin{itemize}
\item[($e$)]$|D_{v}L(x,v)| \leq \alpha(1+|v|)$, \quad $\forall (x,v)\in \mathbb{R}^d\times\mathbb{R}^d$;
\item[($f$)] $|D_{x}L(x,v)| \leq \alpha(1+|v|^{2})$, \quad $\forall (x,v)\in \mathbb{R}^d\times\mathbb{R}^d$;
\item[($g$)]$\frac{1}{4\beta}|v|^{2}- \alpha \leq L(x,v) \leq 4\beta |v|^{2} +\alpha$, \quad $\forall (x,v)\in \mathbb{R}^d\times\mathbb{R}^d$;
\item[($h$)]$\sup\big\{L(x,v): x\in \mathbb{R}^d, |v| \leq R \big\} < +\infty$, \quad $\forall R\geq 0$.
\end{itemize}
\end{remarks}

Define the Hamiltonian $H: \mathbb{R}^{d} \times \mathbb{R}^{d} \to \mathbb{R}$ associated with $L$ by
$$H(x,p)=\sup_{v \in \mathbb{R}^{d}} \Big\{ \big\langle p,v \big\rangle -L(x,v) \Big\}, \quad  \forall (x,p) \in \mathbb{R}^{d} \times \mathbb{R}^{d}.$$
It is straightforward to check that if $L$ is a strict Tonelli Lagrangian, then $H$ satisfies ($a$), ($b$), and ($c$) in \Cref{def:def2}. Such a function $H$ is called a strict Tonelli Hamiltonian.

If $L$ is a {\em reversible} Lagrangian, i.e., $L(x,v)=L(x,-v)$ for all $(x,v) \in \mathbb{R}^{n} \times \mathbb{R}^{n}$, then  $H(x,p)=H(x,-p)$ for all $(x,p) \in \mathbb{R}^{d} \times \mathbb{R}^{d}$.

\medskip

{\it We always work with Tonelli Lagrangians and Hamiltonians, if not stated otherwise.}

\medskip
\noindent $\bullet$ \emph{Invariant measures and holonomic measures}.
The Euler-Lagrange equation associated with $L$
\begin{equation}\label{eq:EL} \frac{d}{dt}D_{v}L(x, \dot x)=D_{x}L(x, \dot x),  \end{equation}
generates a flow of diffeomorphisms $\phi_{t}^{L}: \mathbb{R}^{d} \times \mathbb{R}^{d} \to \mathbb{R}^{d} \times \mathbb{R}^{d}$, with $t \in \mathbb{R}$, defined by
\begin{equation*}\label{lab11} \phi_{t}^{L}(x_{0},v_{0})=( x(t), \dot x(t)), \end{equation*}
where $x: \mathbb{R} \to \mathbb{R}^{d}$ is the maximal solution of \cref{eq:EL} with initial conditions $x(0)=x_{0}, \ \dot x(0)=v_{0}$. It should be noted that, for any Tonelli Lagrangian $L$, the flow $\phi_{t}^{L}$ is complete \cite[Corollary 2.2]{bib:FM}.

We recall that a Borel probability measure $\mu$ on $\mathbb{R}^{d} \times \mathbb{R}^{d}$ is called $\phi_{t}^{L}$-invariant, if $$\mu(B)=\mu(\phi_{t}^{L}(B)), \quad  \forall t \in \mathbb{R}, \quad  \forall B \in \mathcal{B}(\mathbb{R}^{d} \times \mathbb{R}^{d}),$$ or, equivalently,  $$\int_{\mathbb{R}^{d} \times \mathbb{R}^{d}} {f(\phi_{t}^{L}(x,v))\ d\mu(x,v)}=\int_{\mathbb{R}^{d} \times \mathbb{R}^{d}}{f(x,v)\ d\mu(x,v)}, \quad \forall f \in C^{\infty}_{c}(\mathbb{R}^{d} \times \mathbb{R}^{d}).$$ We denote by $\mathfrak{M}_{L}$ the set of all $\phi_{t}^{L}$-invariant Borel probability measures on $\mathbb{R}^{d} \times \mathbb{R}^{d}$.

Let $C^0_{l}$ be the set of all continuous functions $f:\mathbb{R}^d\times \mathbb{R}^d\to \mathbb{R}$ satisfying
\[
\sup_{(x,v)\in \mathbb{R}^d\times \mathbb{R}^d}\frac{|f(x,v)|}{1+|v|}<+\infty
\]
endowed with the topology induced by the uniform convergence on compact subsets.
Denote by $(C^0_{l})'$ the dual of $C^0_{l}$.
Let $\gamma:[0,T]\to \mathbb{R}^d$ be a closed absolutely continuous curve for some $T>0$. Define a probability measure $\mu_{\gamma}$ on the Borel $\sigma$-algebra of $\mathbb{R}^d\times \mathbb{R}^d$ by
\[
\int_{\mathbb{R}^d\times \mathbb{R}^d} fd\mu_{\gamma}=\frac{1}{T}\int_0^Tf(\gamma(t),\dot{\gamma}(t))dt
\]
for all $f\in C^0_{l}$. Let $\mathcal{K}(\mathbb{R}^d)$ denote the set of such $\mu_{\gamma}$'s. We call  $\overline{\mathcal{K}(\mathbb{R}^d)}$ the set of holonomic measures, where $\overline{\mathcal{K}(\mathbb{R}^d)}$ denotes the closure of $\mathcal{K}(\mathbb{R}^d)$ with respect to
the topology induced by the weak convergence on $(C^0_{l})'$. By \cite[2-4.1 Theorem]{bib:Con-b}, we have that
\begin{align}\label{eq:3-40}
\mathfrak{M}_{L}\subseteq \overline{\mathcal{K}(\mathbb{R}^d)}.
\end{align}

\medskip

\noindent $\bullet$ \emph{Ma\~n\'e's critical value}.
If $[a,b]$ is a finite interval with $a<b$ and $\gamma:[a,b]\to \mathbb{R}^d$ is an absolutely continuous curve,
we define its $L$ action as
\[
A_{L}(\gamma)=\int_a^bL(\gamma(s),\dot{\gamma}(s))ds.
\]
The critical value of the Lagrangian $L$, which was introduced by Ma\~n\'e in \cite{bib:Man97}, is defined as follows:
\begin{equation}\label{eq:manedef1}
c_L:=\sup\{k\in\mathbb{R}: A_{L+k}(\gamma)<0 \ \text{for some closed  absolutely continuous curve}\ \gamma\}.
\end{equation}
Since $\mathbb{R}^{d}$ can be seen as a covering of the torus $\mathbb{T}^{d}$, Ma\~n\'e's critical value has the following representation formula~\cite[Theorem A]{bib:CIPP}:
\begin{equation}\label{eq:lab55}
c_L=\inf_{u \in C^{\infty}(\mathbb{R}^{n})} \sup_{ x \in \mathbb{R}^{n}} H(x, Du(x)).
\end{equation}
By \cite[2-5.2 Theorem]{bib:Con-b}, $c_L$ can be also characterized in the following way:
\begin{equation}\label{eq:3-33}
c_L=-\inf\big\{B_L(\nu): \nu\in \overline{\mathcal{K}(\mathbb{R}^d)}\big\}
\end{equation}
where the action $B_{L}$ is defined as
\begin{equation*}
	B_{L}(\nu)=\int_{\R^{d} \times \R^{d}}{L(x,v)\ \nu(dx,dv)}
\end{equation*}
for any $\nu \in \overline{\mathcal{K}(\R^{d})}$. 

If $L$ is a reversible Tonelli Lagrangian, and
\[
\argmin_{x\in\mathbb{R}^d}L(x,0)\neq\emptyset,
\]
then for $x\in \argmin_{x\in\mathbb{R}^d}L(x,0)$, the atomic measure supported at $(x,0)$, $\delta_{(x,0)}$, is a $\phi_{t}^{L}$-invariant probability measure, i.e., $\delta_{(x,0)}\in \mathfrak{M}_{L}$. Note that
\[
B_L(\delta_{(x,0)})\leq B_L(\nu),\quad \forall \nu\in \overline{\mathcal{K}(\mathbb{R}^d)},
\]
which, together with \cref{eq:3-40} and \cref{eq:3-33}, implies that
\begin{equation}\label{eq:3-34}
c_L=-\min\big\{B_L(\nu): \nu\in \mathfrak{M}_L\big\}.
\end{equation}
In view of \cref{eq:3-34}, it is straightforward to see that
\begin{equation}\label{eq:3-35}
c_L=-\min_{x\in\mathbb{R}^d}L(x,0).
\end{equation}

\medskip

\noindent $\bullet$ \emph{Weak KAM theorem}.
Let us recall definitions of weak KAM solutions and viscosity solutions of the Hamilton-Jacobi equation
\begin{align}\label{eq:hj}\tag{HJ}
H(x,Du)=c,
\end{align}
where $c$ is a real number.

\begin{definition}[Weak KAM solutions]\label{def:def3}
A function $u \in C(\mathbb{R}^{d})$ is called a backward weak KAM solution of equation \cref{eq:hj} with $c=c_L$, if it satisfies the following two conditions:
\begin{itemize}
\item[($i$)] for each continuous and piecewise $C^{1}$ curve $\gamma:[t_{1}, t_{2}] \to \mathbb{R}^{d}$, we have that $$u(\gamma(t_{2}))-u(\gamma(t_{1})) \leq \int_{t_{1}}^{t_{2}}{L(\gamma(s), \dot\gamma(s))ds}+c_L(t_{2}-t_{1});$$
\item[($ii$)] for each $x \in \mathbb{R}^{d}$, there exists a $C^{1}$ curve $\gamma:(-\infty,0] \to \mathbb{R}^{d}$ with $\gamma(0)=x$ such that
\[
u(x)-u(\gamma(t))=\int_{t}^{0}{L(\gamma(s), \dot\gamma(s))ds}-c_Lt, \quad  \forall t<0.
\]
 \end{itemize}
\end{definition}
\begin{remarks}\label{rem:re1}
Let $u$ be a function on $\mathbb{R}^d$.
 A $C^1$ curve $\gamma:[a,b]\to \mathbb{R}^d$ with $a<b$ is said to be $(u,L,c_L)$-calibrated, if it satisfies
\[
u(\gamma(t'))-u(\gamma(t))=\int_{t}^{t'}{L(\gamma(s), \dot\gamma(s))ds}+c_L(t'-t), \quad  \forall a\leq t<t'\leq b.
\]
It is not difficult to check that if $u$ satisfies condition ($i$) in \Cref{def:def3}, then the curves appeared in condition ($ii$) in \Cref{def:def3} are necessarily $(u,L,c_L)$-calibrated curves.
\end{remarks}

\begin{definition}[Viscosity solutions]\label{def:visco}
\begin{itemize}
		\item [($i$)] A function $u\in C(\R^{d})$ is called a viscosity subsolution of equation \cref{eq:hj} if for every $\varphi\in C^1(\R^{d})$ at any
local maximum point $x_0$ of $u-\varphi$ on $\R^{d}$ the following holds:
		\[
		H(x_0,D\varphi(x_0))\leq c;
		\]
		\item [($ii$)] A function $u\in C(\R^{d})$ is called a viscosity supersolution of equation \cref{eq:hj} if for every $\varphi\in C^1(\R^{d})$ at any
local minimum point $x_0$ of $u-\varphi$ on $\R^{d}$ the following holds:
		\[
		H(y_0,D\psi(y_0))\geq c;
		\]
		\item [($iii$)] $u$ is a viscosity solution of equation \cref{eq:hj} on $\R^{d}$ if it is both a viscosity subsolution and a viscosity supersolution on $\R^{d}$.
	\end{itemize}
\end{definition}

In \cite{bib:FM} Fathi and Maderna got the existence of backward weak KAM solutions (or, equivalently, viscosity solutions) for $c = c_{L}$.

\section{Hamilton-Jacobi equations with state constraints}
\label{sec:HJstateconstraints}

\subsection{Constrained viscosity solutions.}
Let us recall the notion of constrained viscosity solutions of equation \cref{eq:hj} on $\overline{\Omega}$, see for instance \cite{bib:Soner}.
\begin{definition}[Constrained viscosity solutions]\label{def:scv}
$u\in C(\overline{\Omega})$ is said to be a constrained viscosity solution of \cref{eq:hj} on $\overline{\Omega}$ if it is a subsolution on $\Omega$ and a supersolution on $\overline{\Omega}$.
\end{definition}

Consider the state constraint problem for equation \cref{eq:hj} on $\overline{\Omega}$:
\begin{align}
H(x,Du(x))& \leq c  \quad\text{in}\ \ \Omega,\label{eq:3-50}\\
H(x,Du(x))& \geq c  \quad\text{on}\ \ \overline{\Omega}.\label{eq:3-51}
\end{align}
Mitake \cite{bib:Mit} showed that there exists a unique constant, denoted by $c_H$, such that problem \cref{eq:3-50}-\cref{eq:3-51} admits solutions. Moreover, $c_H$ can be characterized by
\begin{align}\label{eq:3-55}
c_H=\inf\{c\in\mathbb{R}: \cref{eq:3-50}\ \text{has\ a\ solution}\}=\inf_{\varphi \in W^{1,\infty}(\OO)} \esssup_{x \in \OO} H(x, D\varphi(x)).
\end{align}
See \cite[Theorem 3.3, Theorem 3.4, Remark 2]{bib:Mit} for details.

Furthermore, by standard comparison principle for viscosity solutions it is easy to prove that following representation formula for constrained viscosity solutions holds true.

\begin{proposition}[Representation formula for constrained viscosity solutions]\label{prop:for}
$u \in C(\OOO)$ is a constrained viscosity solution of \cref{eq:hj} on $\overline{\Omega}$ for $c=c_H$ if and only if
\begin{equation}\label{eq:LaxOleinik}
u(x)=\inf_{\gamma\in \mathcal{C}(x;t)} \left\{u(\gamma(0))+\int_{0}^{t}{L(\gamma(s), \dot\gamma(s))\ ds} \right\}+c_Ht, \quad \forall x \in \OOO,\ \forall t>0,
\end{equation}
where $\mathcal{C}(x;t)$ denotes the set of all curves $\gamma\in AC([0,t],\OOO)$ with $\gamma(t)=x$.
\end{proposition}

\begin{remarks}
If $u$ is a constrained viscosity solution of \cref{eq:hj} on $\overline{\Omega}$ for $c=c_H$, then by \Cref{prop:for} and \cite[Theorem 5.2]{bib:Mit} one can deduce that $u\in W^{1,\infty}(\Omega)$. Thus, $u$ is Lipschitz in $\Omega$, since $\Omega$ is open and bounded with $\partial \Omega$ of class $C^2$ (see, for instance, \cite[Chapter 5]{bib:Eva}).
\end{remarks}

\begin{proposition}[Equi-Lipschitz continuity of constrained viscosity solutions]\label{prop:equiliplemma}
Let $u$ be a constrained viscosity solution of \cref{eq:hj} on $\overline{\Omega}$ for $c=c_H$. Then, $u$ is Lipschitz continuous on $\OOO$ with a Lipschitz constant $K_1>0$ depending only on $H$.
\end{proposition}

\begin{proof}
Recall that $\OO$ is a bounded domain with $C^{2}$ boundary. By  \Cref{rem:quasi}--(iii) below $\OO$ is $C$-quasiconvex for some $C>0$. Thus, for any $x$, $y\in\OOO$, there is an absolutely continuous curve $\gamma:[0,\tau(x,y)]\to\OOO$ connecting $x$ and $y$, with $0<\tau(x,y)\leq C|x-y|$ and $|\dot{\gamma}(t)|\leq 1$ a.e. in $[0,\tau(x,y)]$. In view of \cref{eq:LaxOleinik}, we deduce that
\[
u(x)-u(y)\leq\int_{0}^{t}L(\gamma(s), \dot{\gamma}(s))\ ds+c_Ht, \quad \forall t>0.
\]
Hence, we get that
\begin{equation*}
u(x)-u(y)\leq C\cdot\Big(\sup_{x\in\OOO,|v|\leq 1}L(x,v)+c_H\Big)\cdot|x-y|	
	\end{equation*}
By exchanging the roles of $x$ and $y$, we get that
\begin{equation*}
|u(x)-u(y)| \leq K_1|x-y|, 	
\end{equation*}
where $K_1:=C\cdot\Big(\sup_{x\in\OOO,|v|\leq 1}L(x,v)+c_H\Big)$ depending only on $H$ and $\OO$.
	\end{proof}

\begin{remarks}\label{rem:quasi}
Let $U\subseteq \mathbb{R}^d$ be a connected open set.

(i) For any $x \in \bar U$ and $C >0$, we say that $y \in \bar U$ is a $(x, C)$-reachable in $U$, if there exists a curve $\gamma \in \text{AC}([0,\tau(x,y)]; \bar U)$ for some $\tau(x,y) > 0$ such that $|\dot\gamma(t)| \leq 1$ a.e. in $[0, \tau(x,y)]$, $\gamma(0)=x$, $\gamma(\tau(x,y))=y$ and $\tau(x,y) \leq C|x-y|$. We denote by $\mathcal{R}_{C}(x,y)$ the set of all $(x,C)$-reachable points from $x \in  U$. We say that $U$ is $C$-quasiconvex if for any $x \in \bar U$ we have that $\mathcal{R}_{C}(x,U)= U$.

(ii) $U$ is called a Lipschitz domain, if $\partial U$ is locally Lipschitz, i.e., $\partial U$ can be locally represented as the graph of a Lipschitz function defined on some open ball of $\R^{d-1}$.

(iii) $U$ is $C$-quasiconvex for some $C>0$ if $U$ is a bounded Lipschitz domain (see, for instance, Sections 2.5.1 and 2.5.2 in \cite{bib:BB1}). Since $\OO$ is a bounded domain with $C^2$ boundary, then it is $C$-quasiconvex for some $C>0$.
\end{remarks}

Consider the assumption

\medskip

\noindent {\bf (A1)} $\displaystyle{\argmin_{x \in \R^{d}}}\ L(x,0) \cap \OOO \not= \emptyset$.

\begin{proposition}\label{prop:criticalvalues}
Let $H$ be a reversible Tonelli Hamiltonian. Assume {\bf (A1)}. Then, $c_H=c_L$.
\end{proposition}

\begin{proof}
By \cite[Theorem 1.1]{bib:FM}, there exists a global viscosity solution $u_{d}$ of equation \cref{eq:hj} with $c=c_L$.
Since $\OO$ is an open subset of $\R^{d}$, by definition $u_{d}\big|_{\OO}$ is solution of \cref{eq:3-50} for $c=c_L$. Thus, $c_H \leq c_L$.

Recalling the characterization \cref{eq:3-55}, since $\varphi$ is differentiable almost everywhere and $H$ is a reversible Tonelli Hamiltonian we have that
\begin{equation*}
c_{H}=\inf_{\varphi \in W^{1,\infty}(\OO)} \esssup_{x \in \OO} H(x, d\varphi(x)) \geq \sup_{x \in \OO} H(x,0) = \max_{x \in \OOO} H(x,0). 	
\end{equation*}
Therefore, by {\bf (A1)} and the fact that
\begin{equation*}
H(x,0)= - \inf_{v \in \R^{d}} L(x,v) \geq -L(x,0)	
\end{equation*}
we deduce that
\begin{equation*}
c_{H} \geq -\min_{x \in \OOO} L(x,0) = - \inf_{x \in \R^{d}} L(x,0) = c_{L},	
\end{equation*}
where the last equality holds by \cref{eq:3-35}. 
\end{proof}

{\it From now on, we assume that $L$ is a reversible Tonelli Lagrangian and denote by $c$ the common value of $c_L$ and $c_H$.}

\begin{remarks}
Comparing to the classical weak KAM solutions, see \Cref{def:def3}, one can call $u: \overline{\Omega} \to \R$ a constrained weak KAM solution if for any $t_{1} < t_{2}$ and any absolutely continuous curve $\gamma: [t_{1}, t_{2}] \to \overline{\Omega}$ we have 
\begin{equation*}
u(\gamma(t_{2}))-u(\gamma(t_{1})) \leq \int_{t_{1}}^{t_{2}}{L(\gamma(s), \dot\gamma(s))\ ds} + c(t_{2}-t_{1})	
\end{equation*}
 and, for any $x \in \overline{\Omega}$ there exists a $C^{1}$ curve $\gamma_{x}: (-\infty, 0] \to \overline{\Omega}$ with $\gamma_{x}(0)=x$ such that 
 \begin{equation*}
 u(x)-u(\gamma_{x}(t)) = \int_{t}^{0}{L(\gamma_{x}(s), \dot\gamma_{x}(s))\ ds}-t, \quad t <0. 
 \end{equation*}
One can easily see that a constrained weak KAM solution must be a constrained viscosity solution by definition and \Cref{prop:for}. In order to prove the opposite relation we need to go back to the $C^{1,1}$ regularity of solutions of the Hamiltonian system associated with a state constraint control problem. 
This will be the subject of a forthcoming paper. 
\end{remarks}

\subsection{Semiconcavity estimates of constrained viscosity solutions}

Here, we give a semiconcavity estimate for constrained viscosity solutions of \cref{eq:1-1}. Note that a similar result has been obtained in \cite[Corollary 3.2]{bib:CCC2} for a general calculus of variation problem under state constraints with a Tonelli Lagrangian and a regular terminal cost. Such regularity of the data allowed the authors to prove the semiconcavity result using the maximum principle which is not possible in our context: indeed, by the representation formula \cref{eq:LaxOleinik}, i.e.
\begin{equation*}
	u(x)=\inf_{\gamma\in\mathcal{C}(x;t)}\left\{u(\gamma(0))+\int^t_0L(\gamma(s),\dot{\gamma}(s))\ ds\right\}+ct,\quad\forall x\in\overline{\Omega},\  \forall t>0
\end{equation*}
one can immediately observe that the terminal cost is not regular enough to apply the maximum principle in \cite[Theorem 3.1]{bib:CCC1}. For these reasons, we decided to prove semiconcavity using a dynamical approach based on the properties of calibrated curves. 

Let $\Gamma^t_{x,y}(\overline{\Omega})$ be the set of all absolutely continuous curve $\gamma:[0,t]\to\overline{\Omega}$ such that $\gamma(0)=x$ and $\gamma(t)=y$. For each $x,y\in\overline{\Omega}$, $t>0$, let
\begin{align*}
	A_t(x,y)=A^{\overline{\Omega},L}_t(x,y)=\inf_{\gamma\in\Gamma^t_{x,y}(\overline{\Omega})}\int^t_0L(\gamma(s),\dot{\gamma}(s))\ ds.
\end{align*}

We recall that, since the boundary of $\Omega$ is of class $C^2$, there exists $\rho_0>0$ such that
\begin{equation}\label{eq:rho_0}
	b_{\Omega}(\cdot)\in C^2_b\ \text{on}\ \Sigma_{\rho_0}=\{y\in B(x,\rho_0): x\in\partial\Omega\}.
\end{equation}

Now, recall a result from \cite{bib:CC}.

\begin{lemma}\label{lem:lem_derivative}
Let $\gamma\in AC([0,T],\R^d)$ and suppose $d_{\Omega}(\gamma(t))<\rho_0$ for all $t\in[0,T]$. Then $d_{\Omega}\circ\gamma\in AC([0,T],\R)$ and
\begin{align*}
	\frac d{dt}(d_{\Omega}\circ\gamma)(t)=\langle Db_{\Omega}(\gamma(t)),\dot{\gamma}(t)\rangle\mathbf{1}_{\Omega^c}(\gamma(t)),\quad a.e., t\in[0,T].
\end{align*}	
\end{lemma}

\begin{proposition}\label{prop:fractional_semiconcavity}
	For any $x^*\in\overline{\Omega}$, the functions $A^{\overline{\Omega},L}_t(\cdot,x^*)$ and $A^{\overline{\Omega},L}_t(x^*,\cdot)$ both are locally semiconcave  with fractional modulus. More precisely, there exists $C>0$ such that, if $h\in\R^d$ with $x\pm h\in\overline{\Omega}$, then we have
	\begin{align*}
		A^{\overline{\Omega},L}_t(x+h,x^*)+A^{\overline{\Omega},L}_t(x-h,x^*)-2A^{\overline{\Omega},L}_t(x,x^*)\leq C|h|^{\frac 32},\quad x\in\overline{\Omega}.
	\end{align*}
	In particular, for such an $h$, we also have
	\begin{align*}
		u(x+h)+u(x-h)-2u(x)\leq C|h|^{\frac 32},\quad x\in\overline{\Omega},
	\end{align*}
	where $u$ is the constrained viscosity solution defined by \cref{eq:1-1}.
\end{proposition}

\begin{proof}
Here we only study the semi-concavity of the function $A^{\overline{\Omega},L}_t(\cdot,x^*)$ for any $x^*\in\overline{\Omega}$, since $A^{\overline{\Omega},L}_t(x^*,\cdot)=A^{\overline{\Omega},\breve{L}}_t(\cdot,x^*)$, where $\breve{L}(x,v)=L(x,-v)$. We divide the proof into several steps.

\noindent\textbf{I. Projection method.} Fix $x,x^*\in\overline{\Omega}$, $h\in\R^d$ such that $x\pm h\in\overline{\Omega}$. Let $\gamma\in\Gamma^t_{x,x^*}(\overline{\Omega})$ be a minimizer for $A_t(x,x^*)$, and $\varepsilon>0$. For $r\in(0,\varepsilon/2]$, define
\begin{align*}
	\gamma_{\pm}(s)=\gamma(s)\pm\left(1-\frac sr\right)_+h,\quad s\in[0,t].
\end{align*}
Recalling \eqref{eq:rho_0}, if $|h|\ll1$, then $d_{\Omega}(\gamma_{\pm}(s))\leq\rho_0$ for $s\in[0,r]$, since
\begin{align*}
	d_{\Omega}(\gamma_{\pm}(s))\leq|\gamma_{\pm}(s)-\gamma(s)|\leq\left|\left(1-\frac sr\right)_+h\right|\leq|h|.
\end{align*}
This implies $d_{\Omega}(\gamma_{\pm}(s))\leq\rho_0$ for all $s\in[0,t]$. Denote by $\widehat{\gamma}_{\pm}$ the projection of $\gamma_{\pm}$ onto $\overline{\Omega}$, that is
\begin{align*}
	\widehat{\gamma}_{\pm}(s)=\gamma_{\pm}(s)-d_{\Omega}(\gamma_{\pm}(s))Db_{\Omega}(\gamma_{\pm}(s)),\quad s\in[0,t].
\end{align*}
From our construction of $\widehat{\gamma}_{\pm}$, it is easy to see that $\widehat{\gamma}_{\pm}(0)=x{\pm}h$ and for all $s\in[0,t]$,
\begin{equation}\label{eq:diff_projection_1}
	\begin{split}
		|\widehat{\gamma}_{\pm}(s)-\gamma(s)|=&\,|\gamma_{\pm}(s)-d_{\Omega}(\gamma_{\pm}(s))Db_{\Omega}(\gamma_{\pm}(s))-\gamma(s)|\\
		\leq &\,|h|+d_{\Omega}(\gamma_{\pm}(s))\leq 2|h|.
	\end{split}
\end{equation}
Moreover, in view of \Cref{lem:lem_derivative}, we conclude that for almost all $s\in[0,r]$
\begin{equation}
	\begin{split}
		\dot{\widehat{\gamma}}_{\pm}(s)=&\,\dot{\gamma}_{\pm}(s)-\left\langle Db_{\Omega}(\gamma_{\pm}(s)),\dot{\gamma}_{\pm}(s)\right\rangle Db_{\Omega}(\gamma_{\pm}(s))\mathbf{1}_{\Omega^c}(\gamma_{\pm}(s))\\
	&\,\quad -d_{\Omega}(\gamma_{\pm}(s))D^2b_{\Omega}(\gamma_{\pm}(s))\dot{\gamma}_{\pm}(s).
	\end{split}
\end{equation}

To proceed with the proof we need estimates for $\int^r_0|\widehat{\gamma}_{+}(s)-\widehat{\gamma}_{-}(s)|^2\ ds$ and $\int^r_0|\dot{\widehat{\gamma}}_{+}(s)-\dot{\widehat{\gamma}}_{-}(s)|^2\ ds$. The first one is easily obtained by \cref{eq:diff_projection_1}. That is
\begin{equation}
	\begin{split}
		\int^r_0|\widehat{\gamma}_{+}(s)-\widehat{\gamma}_{-}(s)|^2\ ds\leq &\,\int^r_0(|\widehat{\gamma}_{+}(s)-\gamma(s)|+|\widehat{\gamma}_{-}(s)-\gamma(s)|)^2\ ds\\
	\leq &\,16r|h|^2.
	\end{split}
\end{equation}
To obtain the second estimate, notice that
\begin{align*}
	&\,\int^r_0|\dot{\widehat{\gamma}}_{+}(s)-\dot{\widehat{\gamma}}_{-}(s)|^2\ ds\\
	\leq &\,\int^r_0(|\dot{\widehat{\gamma}}_{+}(s)-\dot{\gamma}_+(s)|+|\dot{\gamma}_+(s)-\dot{\gamma}_-(s)|+|\dot{\widehat{\gamma}}_{-}(s)-\dot{\gamma}_-(s)|)^2\ ds,
\end{align*}
and
\begin{align*}
	\int^r_0|\dot{\gamma}_+(s)-\dot{\gamma}_-(s)|^2\ ds\leq\frac{4|h|^2}r.
\end{align*}
It is enough to give the estimate of
\begin{align*}
	\int^r_0|\dot{\widehat{\gamma}}_{\pm}(s)-\dot{\gamma}_{\pm}(s)|^2\ ds.
\end{align*}

\noindent\textbf{II. Estimate of $\int^r_0|\dot{\widehat{\gamma}}_{\pm}(s)-\dot{\gamma}_{\pm}(s)|^2\ ds$.} We will only give the estimate for $\int^r_0|\dot{\widehat{\gamma}}_{+}(s)-\dot{\gamma}_{+}(s)|^2\ ds$ since the other is similar. Recalling \Cref{lem:lem_derivative} we conclude that for $s\in[0,r]$,
\begin{align*}
	\dot{\widehat{\gamma}}_{+}(s)-\dot{\gamma}_{+}(s)=&\,-\frac d{ds}\left\{d_{\Omega}(\gamma_{+}(s))Db_{\Omega}(\gamma_{+}(s))\right\}\\
	=&\,-\langle Db_{\Omega}(\gamma_+(s)),\dot{\gamma}_+(s)\rangle Db_{\Omega}(\gamma_{+}(s))\mathbf{1}_{\Omega^c}(\gamma_+(s))\\
	&\,-d_{\Omega}(\gamma_{+}(s))D^2b_{\Omega}(\gamma_{+}(s))\dot{\gamma}_+(s).
\end{align*}
In view of the fact that $\langle D^2b_{\Omega}(x),Db_{\Omega}(x)\rangle=0$ for all $x\in\Sigma_{\rho_0}$, we have that
\begin{align*}
	&\,\int^r_0|\dot{\widehat{\gamma}}_{+}(s)-\dot{\gamma}_{+}(s)|^2\ ds\\
	\leq &\,\int^r_0\langle Db_{\Omega}(\gamma_+(s)),\dot{\gamma}_+(s)\rangle^2\mathbf{1}_{\Omega^c}(\gamma_+(s))\ ds+\int^r_0[d_{\Omega}(\gamma_{+}(s))D^2b_{\Omega}(\gamma_{+}(s))\dot{\gamma}_+(s)]^2\ ds\\
	=&\,I_1+I_2.
\end{align*}
Recall $\gamma\in C^{1,1}([0,T,\overline{\Omega}])$, then $|\dot{\gamma}(s)|$ is uniformly bounded by a constant independent of $h$ and $r$. It follows there exists $C_1>0$ such that
\begin{align*}
	I_2\leq\int^r_0|h|^2\cdot|D^2b_{\Omega}(\gamma_{+}(s))|^2\cdot\left|\dot{\gamma}(s)-\frac hr\right|^2\ ds\leq C_1r|h|^2\left(1+\frac {|h|^2}{r^2}+\frac{|h|}{r}\right).
\end{align*}
In view of \Cref{lem:lem_derivative}, we obtain that
\begin{align*}
	I_1=&\int^r_0\left\{\frac d{ds}(d_{\Omega}(\gamma_+(s)))\langle Db_{\Omega}(\gamma_+(s)),\dot{\gamma}_+(s)\rangle\right\}\mathbf{1}_{\Omega^c}(\gamma_+(s))\ ds.
\end{align*}
We observe that the set $\{s\in[0,r]: \gamma_+(s)\in\overline{\Omega}^c\}$ is open and it is composed of countable union of disjoint open intervals $(a_i,b_i)$. Thus
\begin{align*}
	I_1=\sum_{i=1}^{\infty}\int^{b_i}_{a_i}\left\{\frac d{ds}(d_{\Omega}(\gamma_+(s)))\langle Db_{\Omega}(\gamma_+(s)),\dot{\gamma}_+(s)\rangle\right\}\ ds.
\end{align*}
Integrating by parts we conclude that
\begin{align*}
	I_1=\sum_{i=1}^{\infty}\left\{\frac d{ds}(d_{\Omega}(\gamma_+(s)))d_{\Omega}(\gamma_+(s))\bigg\vert_{a_i}^{b_i}-\int^{b_i}_{a_i}\frac {d^2}{ds^2}(d_{\Omega}(\gamma_+(s)))d_{\Omega}(\gamma_+(s))\ ds\right\}.
\end{align*}
Notice $\gamma_+(a_i),\gamma_+(b_i)\in\partial\Omega$ for all $i\in\N$, so $d_{\Omega}(\gamma_+(a_i))=d_{\Omega}(\gamma_+(b_i))=0$. Moreover, there exists $C_2>0$ such that
\begin{align*}
	\left|\frac {d^2}{ds^2}(d_{\Omega}(\gamma_+(s)))\right|=\left|\frac d{ds}\left\langle Db_{\Omega}(\gamma_+(s)),\dot{\gamma}(s)-\frac hr\right\rangle\right|\leq C_2
\end{align*}
also since $\gamma\in C^{1,1}([0,T,\overline{\Omega}])$. Therefore
\begin{align*}
	I_1\leq C_2r|h|.
\end{align*}
Combing the two estimates on $I_1$ and $I_2$, we have that
\begin{align*}
	\int^r_0|\dot{\widehat{\gamma}}_{+}(s)-\dot{\gamma}_{+}(s)|^2\ ds\leq C_3(r|h|+r|h|^2+\frac{|h|^4}{r}+|h|^3).
\end{align*}

\noindent\textbf{III. Fractional semiconcavity estimate of $A_t(x,x^*)$.} From the previous estimates we have that
\begin{align*}
	&\,\int^r_0|\dot{\widehat{\gamma}}_{+}(s)-\dot{\widehat{\gamma}}_{-}(s)|^2\ ds\\
	\leq &\,3\int^r_0|\dot{\widehat{\gamma}}_{+}(s)-\dot{\gamma}_+(s)|^2+|\dot{\gamma}_+(s)-\dot{\gamma}_-(s)|^2+|\dot{\widehat{\gamma}}_{-}(s)-\dot{\gamma}_-(s)|^2\ ds\\
	\leq &\, C_4(r|h|+r|h|^2+\frac{|h|^4}{r}+|h|^3+\frac{|h|^2}{r}).
\end{align*}

Now, let $\gamma\in\Gamma^t_{x,x^*}(\overline{\Omega})$ be a minimizer for $A_t(x,x^*)$. Thus,
\begin{align*}
	&\,A_t(x+h,x^*)+A_t(x-h,x^*)-2A_t(x,x^*)\\
	\leq &\,\int^r_0\left\{L(\widehat{\gamma}_+(s),\dot{\widehat{\gamma}}_+(s))+L(\widehat{\gamma}_-(s),\dot{\widehat{\gamma}}_-(s))-2L(\gamma(s),\dot{\gamma}(s))\right\}ds\\
	\leq &\,C_5\int^r_0(|\widehat{\gamma}_+(s)-\gamma(s)|^2+|\widehat{\gamma}_-(s)-\gamma(s)|^2+|\dot{\widehat{\gamma}}_+(s)-\dot{\gamma}(s)|^2+|\dot{\widehat{\gamma}}_-(s)-\dot{\gamma}(s)|^2)\ ds.
\end{align*}
Owing to \cref{eq:diff_projection_1}, we have
\begin{align*}
	\int^r_0|\widehat{\gamma}_{\pm}(s)-\gamma(s)|^2\ ds\leq 4r|h|^2.
\end{align*}
On the other hand
\begin{align*}
	&\,\int^r_0|\dot{\widehat{\gamma}}_+(s)-\dot{\gamma}(s)|^2+|\dot{\widehat{\gamma}}_-(s)-\dot{\gamma}(s)|^2\ ds\\
	\leq &\,2\int^r_0|\dot{\widehat{\gamma}}_+(s)-\dot{\gamma}_+(s)|^2+|\dot{\widehat{\gamma}}_-(s)-\dot{\gamma}_-(s)|^2\ ds+C_6\frac{|h|^2}{r}\\
	\leq &\,C_7(r|h|+r|h|^2+\frac{|h|^4}{r}+|h|^3+\frac{|h|^2}{r}).
\end{align*}
Therefore, taking $r=|h|^{\frac 12}$,
\begin{align*}
	&\,A_t(x+h,x^*)+A_t(x-h,x^*)-2A_t(x,x^*)\\
	\leq &\,C_8(r|h|+r|h|^2+\frac{|h|^4}{r}+|h|^3+\frac{|h|^2}{r})\\
	\leq &\,C_9|h|^{\frac 32}.
\end{align*}

\noindent\textbf{IV. Fractional semiconcavity estimate for $u$ defined in \cref{eq:LaxOleinik}.} By using the fundamental solution and fix $t=1$, we have
\begin{align*}
	u(x)=\inf_{y\in\overline{\Omega}}\{u(y)+A_1(y,x)\}+c,\quad\forall x\in\overline{\Omega}.
\end{align*}
Thus the required semiconcavity estimate follow by the relation
\begin{align}\label{eq:semiconc}
	u(x+h)+u(x-h)-2u(x)\leq A_1(y^*,x+h)+A_1(y^*,x-h)-2A_1(y^*,x),
\end{align}
where the infimum above achieves at $y=y^*$.
\end{proof}

\subsection{Differentiability of constrained viscosity solutions}
Let $u$ be a constrained viscosity solution of \cref{eq:hj} on $\overline{\Omega}$.
Recall that, we call $\gamma:[t_1,t_2]\to \overline{\Omega}$ is $(u,L,c)$-calibrated, if it satisfies
\begin{equation*}
u(\gamma(t'_{2}))-u(\gamma(t'_{1})) = \int_{t'_{1}}^{t'_{2}}{L(\gamma(s), \dot\gamma(s))\ ds} + c(t'_{2}-t'_{1}),
\end{equation*}
for any $[t'_{1}, t'_{2}] \subset [t_1,t_2]$.

\begin{proposition}[Differentiability property I]\label{prop:diff}
Let $u$ be a constrained viscosity solution of \cref{eq:hj} on $\overline{\Omega}$. Let $x \in \OO$ be such that there exists a $(u,L,c)$-calibrated curve $\gamma:[-\tau,\tau] \to \OO$ such that $\gamma(0)=x$, for some $\tau > 0$. Then, $u$ is differentiable at $x$.
\end{proposition}


\begin{proposition}[Differentiability property II]\label{prop:tangentialdiff}
Let $u$ be a constrained viscosity solution of \cref{eq:hj} on $\overline{\Omega}$. Let $x \in \partial\OO$ be such that there exists a $(u,L,c)$-calibrated curve $\gamma:[-\tau,\tau] \to \OO$ such that $\gamma(0)=x$, for some $\tau > 0$. For any direction $y$ tangential to $\Omega$ at $x$, the directional derivative of $u$ at $x$ in the direction $y$ exists.
\end{proposition}

We omit the proof of \Cref{prop:diff} here since it follows by standard arguments, as for the case without the constraint $\OO$, and for this we refer to \cite[Theorem 4.11.5]{bib:FA}. Moreover, we prove \Cref{prop:tangentialdiff} in \Cref{sec:appendix}, for the reader convenience,  since it is given by a combination of the arguments in \cite[Theorem 4.11.5]{bib:FA} and the so-called projection method.

Let $x \in \partial\OO$ be such that there exists a $(u,L,c)$-calibrated curve $\gamma:[-\tau,\tau]\to \overline{\Omega}$ with $\gamma(0)=x$ for some $\tau>0$. From \Cref{prop:tangentialdiff}, for any $y$ tangential to $\Omega$ at $x$, we have
\[
\langle D_{v}L(x, \dot{\gamma}(0)),y\rangle=\frac{\partial u}{\partial y}(x).
\]
Thus, one can define the tangential gradient of $u$ at $x$ by
\[
D^{\tau}u(x)=D_{v}L(x, \dot{\gamma}(0))-\langle D_{v}L(x, \dot{\gamma}(0)),\nu(x)
\rangle\nu(x).
\]

Given $x \in \partial\OO$, each $p \in D^{+}u(x)$ can be written as
\[
p=p^{\tau}+p^{\nu},
\]
where $p^{\nu}=\langle p, \nu(x) \rangle \nu(x)$, and $p^{\tau}$ is the tangential component of $p$, i.e.,  $\langle p^{\tau}, \nu(x) \rangle=0$.

By similar arguments to the one in \cite[Proposition 2.5, Proposition 4.3 and Theorem 4.3]{bib:CCC2} and by \Cref{prop:tangentialdiff} it is easy to prove the following result: \Cref{prop:lambda}, \Cref{cor:DD} and \Cref{prop:hamiltonianlemma}.

\begin{proposition}\label{prop:lambda}
Let $x \in \partial\OO$ and $u:\OOO \to \R$ be a Lipschitz continuous and semiconcave function. Then,
\begin{equation*}
-\partial^{+}_{-\nu}u(x)=\lambda_{+}(x):=\max\{\lambda_{p}(x): p \in D^{+}u(x)\},
\end{equation*}
where
\begin{equation*}
\lambda_{p}(x)=\max\{\lambda: p^{\tau}+\lambda\nu(x) \in D^{+}u(x)\},\quad \forall p\in D^+u(x)
\end{equation*}	
and
\begin{equation*}
\partial^{+}_{-\nu}u(x)=\lim_{\substack{h \to 0^{+} \\ \theta \to -\nu \\ x+h\theta \in \overline{\Omega}}} \frac{u(x+h\theta)- u(x)}{h}	
\end{equation*}
denotes the one-sided derivative of $u$ at $x$ in direction $-\nu$. 
\end{proposition}

\begin{corollary}\label{cor:DD}
Let $u$ be a constrained viscosity solution of \cref{eq:hj} on $\overline{\Omega}$. Let $x \in \partial\OO$ be a point such that there is a $(u,L,c)$-calibrated curve $\gamma:[-\tau,\tau] \to \OO$ such that $\gamma(0)=x$, for some $\tau > 0$. Then, all $p\in D^+u(x)$ have the same tangential component, i.e.,
\begin{equation*}
\{ p^{\tau} \in \R^{d}: p \in D^{+}u(x)\}=\{ D^{\tau}u(x)\},
\end{equation*}
and
\begin{equation*}
D^{+}u(x)=\big\{ p \in \R^{d}: p=D^{\tau}u(x)+\lambda \nu(x),\ \forall\ \lambda \in (-\infty, \lambda_{+}(x)]\big\}.
\end{equation*}
\end{corollary}

\begin{proposition}\label{prop:hamiltonianlemma}
Let $u$ be a constrained viscosity solution of \cref{eq:hj} on $\overline{\Omega}$. Let $x \in \partial\OO$ be a point such that there is a $(u,L,c)$-calibrated curve $\gamma:[-\tau,\tau] \to \OO$ such that $\gamma(0)=x$, for some $\tau > 0$. Then,
	\begin{equation*}
	H(x, D^{\tau}u(x)+\lambda_{+}(x)\nu(x))=c.
	\end{equation*}
	\end{proposition}

\begin{theorem}\label{thm:diff3}
Let $u$ be a constrained viscosity solution of \cref{eq:hj} on $\overline{\Omega}$. Let $x \in \OOO$ and $\gamma:[-\tau,\tau]\to \overline{\Omega}$ be a $(u,L,c)$-calibrated curve with $\gamma(0)=x$ for some $\tau >0$. Then,
\begin{itemize}
\item[($i$)] if $x \in \OO$, then
\begin{equation*}
\dot\gamma(0)=D_{p}H(x, Du(x));
\end{equation*}
\item[($ii$)] if $x \in \partial\OO$, then
\begin{equation*}
\dot\gamma(0)=D_{p}H(x, D^{\tau}u(x)+\lambda_{+}(x)\nu(x)).
\end{equation*}
\end{itemize}
\end{theorem}

\begin{proof}
We first prove ($i$). Since $\gamma$ is a calibrated curve, we have that for any $\eps\in(0,\tau]$
\begin{equation*}
u(\gamma(\eps))-u(x)=\int_{0}^{\eps}{L(\gamma(s), \dot\gamma(s))\ ds}+c\eps.
\end{equation*}
Thus, we deduce that
\begin{equation*}
\frac{u(\gamma(\eps))-u(x)}{\eps}-\frac{1}{\eps}\int_{0}^{\eps}{L(\gamma(s), \dot\gamma(s))\ ds}=c,
\end{equation*}
and passing to the limit as $\eps \to 0$, by \Cref{prop:diff} we get
\begin{equation*}
\langle Du(x), \dot\gamma(0) \rangle-L(x, \dot\gamma(0))=c.
\end{equation*}
Since $u$ is a viscosity solution of equation \cref{eq:hj} in $\Omega$, then $c=H(x,Du(x))$. Thus, we deduce that
\begin{equation*}
\langle Du(x), \dot\gamma(0) \rangle-L(x, \dot\gamma(0))=H(x,Du(x)).
\end{equation*}
By the properties of the Legendre transform, the above equality yields
\begin{equation*}
\dot\gamma(0)=D_{p}H(x, Du(x)).
\end{equation*}

In order to prove ($ii$), proceeding as above by \Cref{prop:tangentialdiff} and \Cref{prop:hamiltonianlemma} we get
\begin{equation*}
\langle D^{\tau}u(x), \dot\gamma(0) \rangle-L(x, \dot\gamma(0))=c=H(x, D^{\tau}u(x)+\lambda_{+}(x)\nu(x)).
\end{equation*}
Hence, we obtain that
\begin{equation*}
\dot\gamma(0)=D_{p}H(x, D^{\tau}u(x)+\lambda_{+}(x)\nu(x)).
\end{equation*}
This completes the proof.
\end{proof}

\subsection{Mather set for reversible Tonelli Lagrangians in $\overline{\Omega}$}
Assume {\bf (A1)}. Set
\[
\tilde{\mathcal{M}}_{\OOO}=\{(x,0)\in\OOO\times\mathbb{R}^d\ |\  L(x,0)=\inf_{y\in\mathbb{R}^d}L(y,0)\}.
\]
It is clear that the set $\tilde{\mathcal{M}}_{\OOO}$ is nonempty under assumption {\bf (A1)} and we call $\tilde{\mathcal{M}}_{\OOO}$ the Mather set associated with the Tonelli Lagrangian $L$. Note that
\[
\inf_{x\in\OOO}L(x,0)=\inf_{x\in\mathbb{R}^d}L(x,0)=\inf_{(x,v)\in\mathbb{R}^d\times\mathbb{R}^d}L(x,v),
\]
since $L$ is reversible. Hence, it is straightforward to check that the constant curve at $x$ is a minimizing curve for the action $A_L(\cdot)$, where $x\in \mathcal{M}_{\OOO}:=\pi_1\tilde{\mathcal{M}}_{\OOO}$. We call $\mathcal{M}_{\OOO}$ the projected Mather set.

\begin{definition}[Mather  measures]\label{def:def4}
Let $\mu\in \mathcal{P}(\OOO\times\mathbb{R}^d)$. We say that $\mu$ is a Mather measure for a reversible Tonelli Lagrangian $L$, if $\bar{\mu}\in\mathfrak{M}_L$ and  $\supp(\mu)\subset\tilde{\mathcal{M}}_{\OOO}$, where $\bar{\mu}$ is defined by $\bar{\mu}(B):=\mu\big(B\cap(\OOO\times\mathbb{R}^d)\big)$ for all $B\in \mathcal{B}(\mathbb{R}^d\times\mathbb{R}^d)$.
\end{definition}

\begin{remarks}\label{rem:mini}
Let $x\in\mathcal{M}_{\OOO}$. Let $u$ be a constrained viscosity solution of \cref{eq:hj} on $\overline{\Omega}$.
\begin{itemize}
  \item [(i)] Obviously, the atomic measure $\delta_{(x,0)}$, supported on $(x,0)$, is a Mather measure.
  \item [(ii)] Let $\gamma(t)\equiv x$, $t\in\mathbb{R}$.
Note that $u(\gamma(t'))-u(\gamma(t))=0$ for all $t\leq t'$ and that
\begin{align*}
\int_t^{t'}L(\gamma(s),\dot{\gamma}(s))ds+c(t'-t)=\int_t^{t'}L(x,0)ds+c(t'-t)=0,
\end{align*}
where the last equality comes from \cref{eq:3-35}. Hence, the curve $\gamma$ is a $(u,L,c)$-calibrated curve.
  \item [(iii)] By \Cref{thm:diff3}, we have that
 \begin{align}\label{eq:3-80}
 \begin{split}
\dot\gamma(0)=D_{p}H(x, Du(x)),\quad &\text{if}\ x\in\Omega,\\
\dot\gamma(0)=D_{p}H(x, D^{\tau}u(x)+\lambda_{+}(x)\nu(x)),\quad &\text{if}\ x \in \partial\OO.
\end{split}
\end{align}
  \end{itemize}
\end{remarks}

\begin{proposition}\label{Lip}
Let $u$ be a constrained viscosity solution of \cref{eq:hj} on $\overline{\Omega}$. The function
\[
W: \mathcal{M}_{\OOO}\to \mathbb{R}^d,\quad x\mapsto W(x)
\]
is Lipschitz with a Lipschitz constant depending only on $H$ and $\OO$, where
\begin{align*}
W(x)=
\begin{cases}
Du(x), & \quad \text{if}\ x \in \OO,
\\
D^{\tau}u(x)+\lambda_{+}(x)\nu(x), & \quad \text{if}\ x \in \partial\OO.
\end{cases}
\end{align*}
\end{proposition}

\begin{proof}
In view of \cref{eq:3-80} and the properties of the Legendre Transform, for any $x$, $y\in \mathcal{M}_{\OOO}$ we have
\[
|W(x)-W(y)|=\left|\frac{\partial L}{\partial v}(x,0)-\frac{\partial L}{\partial v}(y,0)\right|\leq K_2|x-y|,
\]
where $K_2>0$ is a constant depending only on $H$ and $\OOO$.
\end{proof}

Let $\mu\in \mathcal{P}(\OOO\times\mathbb{R}^d)$ be minimizing measure and $\mu_1:={\pi_1}\sharp\mu$.
Then $\mu_1$ is a probability measure on $\OOO$ and $\supp(\mu_1)\subset \mathcal{M}_{\OOO}$.
It is clear that
\begin{proposition}\label{prop:1to1}
The map $\pi_1: \supp(\mu)\to\supp(\mu_1)$ is one-to-one and the inverse is given by $x\mapsto \big(x,D_pH(x,W(x))\big)$, where $W(x)$ is as in Proposition \ref{Lip}.
\end{proposition}


\section{Ergodic MFG with  state constraints}
\label{sec:EMFG}

By the results proved so far the good candidate limit system for the MFG system \cref{eq:lab1} is the following 
\begin{equation}\label{eq:MFG1}
\begin{cases}
 H(x, Du)=F(x, m) + \lambda & \text{in} \quad \OOO, \\ \  \text{div}\Big(mV(x)\Big)=0  & \text{in} \quad \OOO,  \\ \ \int_{\OOO}{m(dx)}=1
\end{cases}
\end{equation}
where 
\begin{align*}
V(t,x)=
\begin{cases}
	D_{p}H(x, Du(x)), & x \in \supp(m) \cap \OO,
	\\
	D_{p}H(x, D^{\tau}u(x)+\lambda_{+}(x)\nu(x)), & x \in  \supp(m) \cap \partial\OO
\end{cases}
\end{align*}
and $\lambda_{+}$ is defined in \Cref{prop:lambda}.

\subsection{Assumptions}
From now on, we suppose that $L$ is a reversible strict Tonelli Lagrangian on $\mathbb{R}^d$. Let $F: \mathbb{R}^d \times \mathcal{P}(\OOO) \to \mathbb{R}$ be a function, satisfying the following  assumptions:

\begin{itemize}
	\item[\textbf{(F1)}] for every measure $m \in \mathcal{P}(\OOO)$ the function $x \mapsto F(x,m)$ is of class $C^{2}_{b}(\mathbb{R}^d)$ and	
\begin{equation*}
	\mathcal{F}:=\sup_{m \in \mathcal{P}(\OOO)} \sum_{|\alpha|\leq 2} \| D^{\alpha}F(\cdot, m)\|_{\infty} < +\infty,
	\end{equation*}
	where $\alpha=(\alpha_1,\cdots,\alpha_n)$ and $D^{\alpha}=D^{\alpha_1}_{x_1}\cdots D^{\alpha_n}_{x_n}$;
	\item[\textbf{(F2)}] for every $x \in \mathbb{R}^d$ the function $m \mapsto F(x,m)$ is Lipschitz continuous and
 $${\rm{Lip}}_{2}(F):=\displaystyle{\sup_{\substack{x\in \mathbb{R}^d\\ m_1,\ m_2 \in \mathcal{P}(\OOO) \\ m_1\neq m_2 } }}\frac{|F(x,m_1)-F(x,m_2)|}{d_{1}(m_1, m_2)} < +\infty;$$
	\item[\textbf{(F3)}]  there is a constant $C_F>0$ such that for every $m_{1}$, $m_{2} \in \mathcal{P}(\OOO)$,
	\begin{equation*}
	\int_{\OOO}{(F(x,m_{1})-F(x, m_{2}))\ d(m_{1}-m_{2})} \geq C_F \int_{\OOO}{\left(F(x,m_{1})-F(x,m_{2})\right)^{2}\ dx};
	\end{equation*}

	 \item[\textbf{(A2)}] $\displaystyle{\argmin_{x \in \R^{d}}}\big(L(x,0)+F(x,m)\big)\cap \OOO\neq\emptyset$, $\quad \forall m\in  \mathcal{P}(\OOO)$.
	\end{itemize}

\medskip\noindent
Note that assumption {\bf (A2)} is MFG counterpart of assumption {\bf (A1)} which guarantees, as we will see, that for any measure $m$ the Mather set associated with the Lagrangian $L(x,v)+F(x,m)$ is non-empty.

\begin{definition}[Solutions of constrained ergodic  MFG system]\label{def:ers}
A triple $$(\bar\lambda, \bar u, \bar m) \in \mathbb{R} \times C(\OOO) \times \mathcal{P}(\OOO)$$ is called a solution of constrained ergodic MFG system \cref{eq:MFG1} if
	\begin{itemize}
		\item[($i$)] $\bar u$ is a constrained viscosity solution of the first equation of system \cref{eq:MFG1};
		\item[($ii$)] $D\bar u$ exists for $\bar m-a.e. \ \ x \in \OOO$;
		\item [($iii$)] $\bar m$ is a projected minimizing measure, i.e., there is a minimizing measure $\eta_{\bar m}$ for $L_{\bar m}$ such that $\bar m={\pi_{1}}\sharp \eta_{\bar m}$;
		\item[($iv$)] $\bar m$ satisfies the second equation of system \cref{eq:MFG1} in the sense of distributions, that is,
		$$ \int_{\OOO}{\big\langle Df(x), V(x) \big\rangle\ d\bar m(x)}=0, \quad  \forall f \in C_c^{\infty}(\OOO),
		$$
where the vector field $V$ is related to $\bar u$ in the following way: if $x\in\OO\cap\supp(\bar m)$, then $D\bar u(x)$ exists and
\[
V(x)=D_pH(x,D\bar u(x));
\]
if $x\in\partial\OO\cap\supp(\bar m)$, then
\[
V(x)=D_pH(x,D^{\tau}\bar{u}(x)+\lambda_{+}(x)\nu(x)).
\]
	\end{itemize}
\end{definition}
We denote by $\mathcal{S}$ the set of  solutions of system \cref{eq:MFG1} and \Cref{thm:MR1} below guarantees the nonemptiness of such set.

\begin{definition}[Mean field Lagrangians and Hamiltonians]\label{def:mfgl}
Let $H$ be the reversible strict Tonelli Hamiltonian associated with $L$.
For any  $m\in \mathcal{P}(\OOO)$, define the mean field Lagrangian and Hamiltonian associated with $m$
by
\begin{align}
L_{m}(x,v)&:=L(x,v)+F(x,m),\,\quad (x,v)\in\mathbb{R}^d\times\mathbb{R}^d,\label{eq:lm}\\
H_{m}(x,p)&:=H(x,p)-F(x,m),\quad (x,p)\in\mathbb{R}^d\times\mathbb{R}^d\label{eq:hm}.
\end{align}
\end{definition}
By assumption {\bf (F1)}, it is clear that for any given $m\in \mathcal{P}(\OOO)$,  $L_m$ (resp. $H_{m}$) is a reversible strict Tonelli Lagrangian (resp. Hamiltonian). So, in view of {\bf (A2)}, all the results recalled and proved in \Cref{sec:HJstateconstraints} still hold for $L_m$ and $H_m$.

In view of \Cref{prop:criticalvalues}, for any given $m\in \mathcal{P}(\OOO)$, we have $c_{H_{m}}=c_{L_m}$. Denote the common value of $c_{H_{m}}$ and $c_{L_m}$ by $\lambda(m)$.

\begin{lemma}[Lipschitz continuity of the critical value]\label{lem:LEM1}
The function $m\mapsto \lambda(m)$ is Lipschitz continuous on $\mathcal{P}(\OOO)$ with respect to the metric $d_{1}$, where the Lipschitz constant depends on $F$ only.
\end{lemma}

Since the characterization \cref{eq:3-55} holds true, the proof of this result is an adaptation of \cite[Lemma 1]{bib:CCMW}.

\bigskip

\subsection{Existence of solutions of constrained ergodic  MFG systems}
We are now in a position to prove $\mathcal{S}\neq\emptyset$.
\begin{theorem}[Existence of solutions of \cref{eq:MFG1}]\label{thm:MR1}
	Assume {\bf (F1)}, {\bf (F2)}, and {\bf (A2)}.
	\begin{itemize}
		\item [($i$)] There exists at least one solution $(c_{H_{\bar m}},\bar u,\bar m)$ of system \cref{eq:MFG1}, i.e., $\mathcal{S}\neq\emptyset$.
		\item [($ii$)] Assume, in addition, {\bf (F3)}. Let
	$(c_{H_{\bar m_{1}}}, \bar u_{1}, \bar m_{1})$, $(c_{H_{\bar m_{2}}}, \bar u_{2}, \bar m_{2})\in \mathcal{S}$. Then,
	\[
	F(x,\bar m_{1})= F(x,\bar m_{2}),\quad \forall x\in \OOO\quad  \text{and}\quad  c_{H_{\bar m_{1}}}=c_{H_{\bar m_{2}}}.
	\]
\end{itemize}
\end{theorem}
	
\begin{remarks}
By ($ii$) in \Cref{thm:MR1}, it is clear that each element of $\mathcal{S}$ has the form $(\bar \lambda,\bar u,\bar m)$, where $\bar m$ is a projected minimizing measure and $\bar\lambda$ denotes the common Ma\~n\'e critical value of $H_{\bar m}$. 
\end{remarks}	

\proof[Proof of \Cref{thm:MR1}]
	The existence result $(i)$ follows by the application of the Kakutani fixed point theorem. Indeed, by the arguments in Section 3, for any $m\in\mathcal{P}(\OOO)$, there is a minimizing measure $\eta_{m}$ associated with $L_m$. Thus we can define a set-valued map as follows
$$
\Psi: \PP(\OOO) \rightarrow \PP(\OOO), \quad m \mapsto \Psi(m)
$$
where
\[
\Psi(m):=\left\{{\pi_{1}}\sharp \eta_{m}:\ \eta_{m}\ \text{is a minimizing measure for}\ L_m \right\}.
\]
	Then, a fixed point $\bar{m}$ of $\Psi$ is a solution in the sense of distributions of the stationary continuity equation and there exists a constrained viscosity solution associated with $H_{\bar{m}}$ by \cite{bib:Mit}. For more details see for instance \cite[Theorem 3]{bib:CCMW}.  
	
Let
	$(c_{H_{\bar m_{1}}}, \bar u_{1}, \bar m_{1})$, $(c_{H_{\bar m_{2}}}, \bar u_{2}, \bar m_{2})\in \mathcal{S}$. Given any $T>0$,
define the following sets of curves:
\[
\Gamma:=\{\gamma\in AC([0,T];\mathbb{R}^d): \gamma(t)\in\OOO, \ \forall t\in[0,T]\},
\]
and
\[
\tilde{\mathcal{M}}_{\bar{m}_i}:=\big\{\text{constant curves}\ \gamma:[0,T]\to\OOO,\ t\mapsto x:\ x\in \argmin_{y\in\OOO}L_{\bar{m}_i}(y,0)\big\},\quad i=1,2.
\]
One can define Borel probability measures on $\Gamma$ by
\begin{align*}
\mu_i(\tilde{B})=
\begin{cases}
\bar{m}_1(B), & \quad \tilde{B}\cap \tilde{\mathcal{M}}_{\bar{m}_i}\neq\emptyset,
\\
0, & \quad \text{otherwise},
\end{cases}
\end{align*}
where
\[
B=\{x\in\OOO: \text{the constant curve}\ t\mapsto x\ \text{belongs to}\ \tilde{B}\cap \tilde{\mathcal{M}}_{\bar{m}_i}\}.
\]
By definition, it is direct to see that $\supp(\mu_i)\subset \tilde{\mathcal{M}}_{\bar{m}_i}$ and
\begin{align}\label{eq:4-200}
\bar{m}_i=e_t\sharp\mu_i,\quad \forall t\in[0,T].
\end{align}
Given any $x_0\in\supp(\bar{m}_1)$, let $\gamma_1$ denote the constant curve $t\mapsto x_0$, then for any $t>0$ we have that
\begin{align*}
0=\bar{u}_1(x_0)-\bar{u}_1(\gamma_1(0))=\int_0^t\big(L(\gamma_1,\dot{\gamma}_1)+F(\gamma_1,\bar{m}_1)\big)ds+c_{H_{\bar m_{1}}}t,\\
0=\bar{u}_2(x_0)-\bar{u}_2(\gamma_1(0))\leq\int_0^t\big(L(\gamma_1,\dot{\gamma}_1)+F(\gamma_1,\bar{m}_2)\big)ds+c_{H_{\bar m_{2}}}t,
\end{align*}
which imply that
\[
\int_0^t\big(F(\gamma_1,\bar{m}_1)-F(\gamma_1,\bar{m}_2)\big)ds+(c_{H_{\bar m_{1}}}-c_{H_{\bar m_{2}}})t\leq 0.
\]
By integrating the above inequality on $\Gamma$ with respect to $\mu_1$, we get that
\[
\int_\Gamma\int_0^t\big(F(\gamma_1,\bar{m}_1)-F(\gamma_1,\bar{m}_2)\big)dsd\mu_1+(c_{H_{\bar m_{1}}}-c_{H_{\bar m_{2}}})t\leq 0.
\]
In view of Fubini Theorem and \cref{eq:4-200}, we deduce that
\[
\int_0^t\int_{\OOO}\big(F(x,\bar{m}_1)-F(x,\bar{m}_2)\big)d\bar{m}_1ds+(c_{H_{\bar m_{1}}}-c_{H_{\bar m_{2}}})t\leq 0,
\]
implying that
\[
\int_{\OOO}\big(F(x,\bar{m}_1)-F(x,\bar{m}_2)\big)d\bar{m}_1+(c_{H_{\bar m_{1}}}-c_{H_{\bar m_{2}}})\leq 0.
\]
Exchanging the roles of $\bar{m}_1$ and $\bar{m}_2$, we obtain that
\[
\int_{\OOO}\big(F(x,\bar{m}_2)-F(x,\bar{m}_1)\big)d\bar{m}_2+(c_{H_{\bar m_{2}}}-c_{H_{\bar m_{1}}})\leq 0.
\]
Hence, we get that
\[
\int_{\OOO}\big(F(x,\bar{m}_2)-F(x,\bar{m}_1)\big)d(\bar{m}_2-\bar m_{1})\leq 0.
\]
Recalling assumption ({\bf{F3}}), we deduce that $F(x,\bar{m}_2)=F(x,\bar{m}_1)$ for all $x\in\OOO$ and thus $c_{H_{\bar m_{1}}}=c_{H_{\bar m_{2}}}$.
\qed

\begin{remarks}\label{rem:stress1}
	Note that even though the uniqueness result is a consequence of the classical Lasry--Lions monotonicity condition for MFG system, our proof here differs from the one in \cite{bib:CCMW} and in \cite{bib:CAR}: indeed, in our setting the stationary continuity equation has different vector fields depending on the mass of the measure in $\OO$ and the mass on $\partial \OO$. This is why we addressed the problem representing the Mather measures associated with the system through measures supported on the set of calibrated curves. 
\end{remarks}

\section{Convergence of mild solutions of the constrained MFG problem}
\label{sec:convergence}
This section is devoted to the long-time behavior of first-order constrained MFG system \cref{eq:lab1}.
We will assume {\bf (F1)}, {\bf (F2)}, {\bf (F3)},  {\bf (A2)}, and the following additional conditions:
\begin{itemize}
	\item [\textbf{(U)}] $u^f\in C^1_b(U)$, where $U$ is an open subset of $\mathbb{R}^d$ such that $\OOO\subset U$.

	\item [\textbf{(A3)}] the set-value map $(\PP(\OOO), d_{1}) \longrightarrow (\R^{d}, |\cdot|)$ such that $$m \mapsto \argmin_{x \in \OOO} \{L(x,0)+F(x,m)\}$$ has a Lipschitz selection, i.e. $\xi_{*}(m) \in \argmin_{x \in \OOO}\{L(x,0) + F(x,m)\}$ and moreover, for any $m \in \PP(\OOO)$ $$\min_{x \in \OOO} \left\{L(x,0)+F(x,m) \right\}=0.$$
\end{itemize}

\begin{remarks}\label{rem:re5-100}
Assumptions {\bf (A2)} and  {\bf (A3)} imply that $L(x,v)+F(x,m)\geq 0$ for all $(x,v)\in \mathbb{R}^d\times\mathbb{R}^d$ and all $m\in\mathcal{P}(\OOO)$. Moreover, if $(u^{T}, m^{\eta})$ is a mild solution of the MFG system \cref{eq:lab1} then by assumption {\bf (A3)} we have that there exists a Lipschitz continuous curve $\xi_{*}: [0,T] \to \OOO$ such that $$\xi_{*}(t) \in \argmin_{x \in \OOO}\{L(x,0)+F(x, m^{\eta}_{t})\}$$ and $L(\xi_{*}(t), \dot\xi_{*}(t))+F(\xi_{*}(t), m^{\eta}_{t})=0$ for any $t \in [0,T]$.
\end{remarks}
We provide now two examples of mean field Lagrangian that satisfies assumption {\bf (A3)}. 
\begin{example}\em
	\begin{enumerate}
	\item 	Let $L_{m}(x,v)=\frac{1}{2}|v|^{2}+f(x)g(m)$ for some continuous functions $f: \OOO \to \R$ and $g: \PP(\OOO) \to \R_{+}$. Then,  we have that $$\argmin_{x \in \OOO}\{L(x,0)+f(x)g(m)\}=\argmin_{x \in \OOO} f(x)g(m), $$ $\bar{x} =\displaystyle{\min_{x \in \OOO}}\ f(x)g(m)$ is unique and doesn't depend on $m \in \PP(\OOO)$.  Thus, the Lipschitz selection of minimizers of the mean field Lagrangian is the constant one, i.e. $\xi_{*}(m) \equiv \bar{x}$. 
	\item Let $L_{m}(x,v)= \frac{1}{2}|v|^{2} + \big(f(x) + g(m) \big)^{2}$, where $g: \PP(\OOO) \to \R$ is Lipschitz continuous with respect to the $d_{1}$ distance and $f: \OOO \to \R$ is such that $f^{-1}$ is Lipschitz continuous. Thus the minimum is  reached at $f(x)=-g(m)$ and by the assumptions on $f$ and $g$ such minimum has a Lipschitz depends with respect to $m \in \PP(\OOO)$. 
	\end{enumerate}
\end{example}

\subsection{Convergence of mild solutions}

In order to get the convergence result of mild solutions of system \cref{eq:lab1}, we prove two preliminary results first.
\begin{lemma}[Energy estimate]\label{lem:energyestimate}
 There exists a constant $\bar\kappa \geq 0$ such that for any mild solution $(u^{T}, m^{\eta}_t)$ of constrained MFG system \cref{eq:lab1} associated with a constrained MFG equilibrium $\eta \in \PP_{m_{0}}(\Gamma)$, and any solution $(\bar{u}, \bar\lambda, \bar{m})$ of constrained ergodic  MFG system \cref{eq:MFG1}, there holds\begin{equation*}
\int_{0}^{T}\int_{\OOO}{\Big(F(x, m^{\eta}_{t})-F(x, \bar{m}) \Big)\ \big(m^{\eta}_{t}(dx)-\bar{m}(dx)\big)dt}\leq \bar\kappa,
\end{equation*}
where $\bar \kappa$ depends only on $L$, $F$ and $\OO$.
\end{lemma}

\begin{proof}
As we did in the proof of \Cref{thm:MR1} (ii), one can define a Borel probability measure on $\Gamma$ by
\begin{align*}
\bar\eta(\tilde{B})=
\begin{cases}
\bar{m}(B), & \quad \tilde{B}\cap \tilde{\mathcal{M}}_{\bar{m}}\neq\emptyset,
\\
0, & \quad \text{otherwise},
\end{cases}
\end{align*}
where
\[
B=\{x\in\OOO: \text{the constant curve}\ t\mapsto x\ \text{belongs to}\ \tilde{B}\cap \tilde{\mathcal{M}}_{\bar{m}}\},
\]
and
\[
\tilde{\mathcal{M}}_{\bar{m}}:=\big\{\text{constant curves}\ \gamma:[0,T]\to\OOO,\ t\mapsto x:\ x\in \argmin_{y\in\OOO}L_{\bar{m}}(y,0)\big\}.
\]
By definition, it is direct to check that $\supp(\bar\eta)\subset \tilde{\mathcal{M}}_{\bar{m}}$ and
\begin{align}\label{eq:4-270}
\bar{m}=e_t\sharp\bar\eta,\quad \forall t\in[0,T].
\end{align}
	
Note that
\begin{align*}
& \int_{0}^{T}\int_{\OOO}{\Big(F(x,m^{\eta}_{t})-F(x, \bar{m}) \Big)\ (m^{\eta}_{t}(dx)-\bar{m}(dx))dt}
\\
=& \underbrace{\int_{0}^{T}\int_{\Gamma^{*}_{\eta}}{\Big(F(\gamma(t), m^{\eta}_{t})-F(\gamma(t), \bar{m})\Big)\ d\eta(\gamma)dt}}_{\bf A}
\\
-&  \underbrace{\int_{0}^{T}\int_{\tilde{\mathcal{M}}_{\bar{m}}}{\Big(F(\bar\gamma(t), m^{\eta}_{t})-F(\bar\gamma(t), \bar{m}) \Big)\ d\bar\eta(\bar\gamma)dt}}_{\bf B},
\end{align*}
where $\Gamma^{*}_{\eta}$ is as in \Cref{def:equili}. First, we consider term {\bf A}:
\begin{align*}
\textbf{A}=&\int_{0}^{T}\int_{\Gamma^{*}_{\eta}}{\Big(F(\gamma(t), m^{\eta}_{t})-F(\gamma(t), \bar{m}) \Big)\ d\eta(\gamma)dt}\\
= & \int_{\Gamma^{*}_{\eta}}\int_{0}^{T}{\Big(L(\gamma(t), \dot\gamma(t))+ F(\gamma(t), m^{\eta}_{t})\Big)\ dt d\eta(\gamma)} 
\\
- & \int_{\Gamma^{*}_{\eta}}\int_{0}^{T}{\Big(L(\gamma(t), \dot\gamma(t))+F(\gamma(t), \bar{m}) \Big) dt d\eta(\gamma)}.
\end{align*}
Since $\eta$ is a constrained MFG equilibrium associated with $m_0$, then any curve $\gamma\in\supp(\eta)$ satisfies the following equality
\[
u^{T}(0,\gamma(0))-u^{f}(\gamma(T))=\int_{0}^{T}\Big(L(\gamma(t), \dot\gamma(t))+ F(\gamma(t), m^{\eta}_{t})\Big)\ dt.
\]
In view of \cref{eq:LaxOleinik} with $L=L_{\bar m}$, one can deduce that
\[
\bar u(\gamma(T))-\bar u(\gamma(0))\leq \int_{0}^{T}\Big(L(\gamma(t), \dot\gamma(t))+F(\gamma(t),\bar{m}) \Big)dt+\lambda(\bar{m})T.
\]
Hence, we have that
\begin{align*}
\textbf{A}\leq  \int_{\Gamma^{*}_{\eta}}{\Big(u^{T}(0,\gamma(0))-u^{f}(\gamma(T)) \Big)\ d\eta(\gamma)}
	+ \int_{\Gamma^{*}_{\eta}}{\Big(\bar u(\gamma(0))-\bar u(\gamma(T)) \Big)\ d\eta(\gamma)}+\lambda(\bar{m})T.
	\end{align*}
By \Cref{prop:equiliplemma} we estimate the second term of the right-hand side of the above inequality as follows:
	\begin{equation*}
	\int_{\Gamma^{*}_{\eta}}{\Big(\bar{u}(\gamma(0))-\bar{u}(\gamma(T)) \Big)\ d\eta(\gamma)} \leq K_{2}\cdot\int_{\Gamma^{*}_{\eta}}{|\gamma(0)-\gamma(T)|\ d\eta(\gamma)} \leq K_{2}\cdot\text{diam}(\OOO),
	\end{equation*}
	where $K_{2}:=C\cdot\Big( \sup_{x\in\OOO, |v|\leq 1}L(x,v)+\mathcal{F}+\sup_{m\in \mathcal{P}(\OOO)}\lambda(m)\Big)$ comes from \Cref{prop:equiliplemma}. In view of \Cref{lem:LEM1} and the compactness of $\mathcal{P}(\OOO)$, $K_2$ is well-defined and depends only on $L$, $F$ and $\OO$.
For the first term, since $u^{f}$ is bounded on $\OOO$, we only need to estimate $u^{T}(0, \gamma(0))$ where $\gamma \in \supp(\eta)$.

Take $\xi_{*}$ as in assumption {\bf (A3)} and  \Cref{rem:re5-100}.
Since $\Omega$ is $C$-quasiconvex, there is $\beta \in \Gamma$ such that $\beta(0)=\gamma(0)$, $\beta(\tau(\gamma(0), \xi_{*}(0)))=\xi_{*}(0)$ and $|\dot\beta(t)| \leq 1$ a.e. in $t \in [0, \tau(\gamma(0), \xi_{*}(0))]$, where $\tau(\gamma(0), \xi_{*}(0))\leq C|\gamma(0)-\xi_{*}(0)|$.
Define a curve $\xi\in\Gamma$ as follows:
\begin{align*}
	&\text{if}\ T<\tau(\gamma(0), \xi_{*}(0)),\quad \xi(t)=\beta(t),\ \ \ \qquad t\in[0,T];\\
 &\text{if}\ T\geq \tau(\gamma(0), \xi_{*}(0)),\quad \xi(t)=
\begin{cases}
	\beta(t), & \quad t \in [0, \tau(\gamma(0), \xi_{*}(0))],
	\\
	\xi_{*}(t), & \quad t \in (\tau(\gamma(0), \xi_{*}(0)), T].
\end{cases}	
\end{align*}

If $T\geq \tau(\gamma(0),\xi_{*}(0))$, since $(u^T,m^{\eta}_t)$ is mild solution of \cref{eq:lab1}, we deduce that
\begin{align*}
& u^{T}(0,\gamma(0))  \leq \int_{0}^{\tau(\gamma(0), \xi_{*}(0))}{\Big(L(\beta(t), \dot\beta(t)) + F(\beta(t), m^{\eta}_t) \Big)\ dt} 
\\
+ &  \int_{\tau(\gamma(0), \xi_{*}(0))}^{T}{\Big(L(\xi_{*}(t), \dot\xi_{*}(t)) + F(\xi_{*}(t), m^{\eta}_t) \Big)\ dt}+u^f(\xi_{*}).
\end{align*}
Thus, by {\bf (A3)} we have that the second integral of the right-hand side of the above inequality is zero. Hence,
\begin{align*}
u^{T}(0,\gamma(0))  &\leq \int_{0}^{\tau(\gamma(0), \xi_{*}(0))}{\Big(L(\beta(t), \dot\beta(t))+ F(\beta(t), m^{\eta}_t) \Big)\ dt}+u^f(\bar x)\\
 &\leq \Big( \max_{\substack{ y \in \OOO \\ |v| \leq 1}}\big(L(y,v)+\mathcal{F}\big)  \Big)\cdot\tau(\gamma(0), \bar{x})+\|u^f\|_\infty
\\
&\leq \Big( \max_{\substack{ y \in \OOO \\ |v| \leq 1}}\big(L(y,v)+\mathcal{F}\big)   \Big)\cdot C\cdot \text{diam}(\OOO)+\|u^f\|_\infty.
\end{align*}
We can conclude that
\begin{equation}\label{eq:ineA}
{\bf A} \leq K_{2}\cdot\text{diam}(\OOO) + \|u^{f} \|_{\infty} + 	\Big( \max_{\substack{ y \in \OOO \\ |v| \leq 1}}\big(L(y,v)+F(x,\bar{m}) \big)\Big)\cdot C\cdot  \text{diam}(\OOO) + c(H_{\bar{m}})T.
\end{equation}

If $T<\tau(\gamma(0),\xi_{*}(0))$, in view if \Cref{rem:re5-100}, we get that
\begin{align*}
u^{T}(0,\gamma(0))  &\leq \int_{0}^{T}{\Big(L(\beta(t), \dot\beta(t)) + F(\beta(t), m^{\eta}_t) \Big)\ dt} +u^f(\xi(T))\\
&\leq
\int_{0}^{\tau(\gamma(0), \xi_{*}(0))}{\Big(L(\beta(t), \dot\beta(t)) + F(\beta(t), m^{\eta}_t) \Big)\ dt}+u^f(\xi(T)).
\end{align*}
Thus, we can get \cref{eq:ineA} again.

Now we estimate term {\bf B}. Note that
\begin{align*}
\textbf{B}=& \int_{0}^{T}\int_{\tilde{\mathcal{M}}_{\bar{m}}}{\Big(F(\bar{\gamma}(t), m^{\eta}_{t})-F(\bar{\gamma}(t), \bar{m}) \Big)\ d\bar\eta(\bar{\gamma})dt}
\\
=& \int_{\tilde{\mathcal{M}}_{\bar{m}}}\int_{0}^{T}{\Big(L(\bar\gamma(t), \dot{\bar{\gamma}}(t))+F(\bar\gamma(t), m^{\eta}_t) \Big)\ dtd\bar\eta(\bar\gamma)}
\\
- & \int_{\tilde{\mathcal{M}}_{\bar{m}}}\int_{0}^{T}{\Big(L(\bar\gamma(t), \dot{\bar{\gamma}}(t))+F(\bar\gamma(t), \bar{m}) \Big)\ dtd\bar\eta(\bar\gamma)}
\\
\geq & \int_{\tilde{\mathcal{M}}_{\bar{m}}}{\Big(u^{T}(0, \bar\gamma(0))-u^{f}(\bar\gamma(T)) \Big)\ d\bar\eta(\bar\gamma)} 
\\
+ & \int_{\tilde{\mathcal{M}}_{\bar{m}}}{\Big(\bar{u}(\bar\gamma(0))-\bar{u}(\bar\gamma(T)) \Big)\ d\bar\eta(\bar\gamma)}+c(H_{\bar{m}})T.
\end{align*}
By \Cref{rem:re5-100}, we obtain that
\begin{equation*}
	\int_{\tilde{\mathcal{M}}_{\bar{m}}}{\Big(u^{T}(0, \bar\gamma(0))-u^{f}(\bar\gamma(T)) \Big)\ d\bar\eta(\bar\gamma)} \geq - 2\| u^{f} \|_{\infty},
\end{equation*}
and since $\bar\gamma \in \supp(\bar\eta)$ is a constant curve we deduce that
\begin{equation*}
	\int_{\tilde{\mathcal{M}}_{\bar{m}}}{\Big(\bar{u}(\bar\gamma(0))-\bar{u}(\bar\gamma(T)) \Big)\ d\bar\eta(\bar\gamma)}=0.
\end{equation*}
Hence, we have that
\begin{equation}\label{eq:ineB}
-\textbf{B} \leq 2\| u^{f}\|_{\infty}-c(H_{\bar{m}})T.	
\end{equation}
Therefore, combining \cref{eq:ineA} and \cref{eq:ineB} we conclude that
\begin{align*}
	& \int_{0}^{T}\int_{\OOO}{\Big(F(x, m^{\eta}_{t})-F(x, \bar{m}) \Big)\ \big(m^{\eta}_{t}(dx)-\bar{m}(dx)\big)dt}
	\\
	 \leq & 3 \|u^{f} \|_{\infty} + 	\left( C\cdot\max_{\substack{ y \in \OOO \\ |v| \leq 1}}\big(L(y,v)+\mathcal{F}\big)+ K_{2} \right)\cdot \text{diam}(\OOO).
\end{align*}
\end{proof}

\begin{lemma}\label{lem:leA}
Let $f: \OOO \to \R$ be a Lipschitz continuous function with Lipschitz constant $\text{Lip}[f] \leq M$ for some constant $M \geq 0$. Then, there exists a constant $C(d,M)\geq 0$ such that
\begin{equation*}
\| f \|_{\infty} \leq C(d,M)\|f\|_{2, \Omega}^{\frac{2}{2+d}}.	
\end{equation*}
\end{lemma}

\begin{proof}
Fix $x_{0} \in \partial\OO$ and fix a radius $r \geq 0$. We divide the proof into two parts: first, we assume that $\OOO$ coincides with the half-ball centered in $x_{0}$ with radius $r$ contained in $\{x \in \R^{d}: x_{d} \geq 0\}$ such that $x_{0} \in \{x \in \R^{d}: x_{d}=0\}$; then, we remove this constraint proving the result for a general domain $\OO$.

\underline{Part I}: We denote by $B^{+}$ the set $\overline{B}_{r}(x_{0}) \cap \{x \in \R^{d}: x_{d} \geq 0\}$ and by $B^{-}$ the complement of $B^{+}$.
Let $\tilde{f}$ denote the following extension of $f$ in $\overline{B}_{r}(x_{0})$:
\begin{align*}
\tilde{f}(x)=
\begin{cases}
	f(x), & \quad \text{if}\ x \in B^{+}
	\\
	f(x_{1}, \dots, x_{d-1}, -x_{d}), & \quad \text{if}\ x \in B^{-}.
\end{cases}	
\end{align*}

Let $\chi_{r}$ denote a cut-off function such that $\chi_{r}(x)=1$ for $x \in B_{r}(x_{0})$, $\chi_{r}(x)=0$ for $x \in \R^{d} \backslash B_{2r}(x_{0})$ and $0 \leq \chi_{r}(x) \leq 1$ for $x \in B_{2r}(x_{0}) \backslash B_{r}(x_{0})$ and let $\tilde{f}_{r}$ be the extension of $\tilde{f}$ on $\R^{d}$, i.e. $\tilde{f}_{r}(x):= \tilde{f}\cdot \chi_{r}(x)$. Moreover, for any $\delta > 0$ we consider a cover of $\overline{B}_{r}(x_{0})$ through cubes of length $\delta$ denoted by $Q_{\delta}$.
Then, by construction we have that for any cube $Q_{\delta}$
\begin{equation*}
	\| \tilde{f}_{r}\|_{2, Q_{\delta}} \leq C(\delta) \| f\|_{2, B^{+}},
\end{equation*}
for some constant $C(\delta) \geq 0$. Therefore, applying Lemma 4 in \cite{bib:CCMW} we get
\begin{equation}\label{eq:est1}
\| f \|_{\infty, B^{+}} \leq \| \tilde{f}_{r} \|_{\infty, Q_{\delta}} \leq C(\delta, M) \| \tilde{f}_{r}\|_{2, Q_\delta}^{\frac{2}{d+2}} \leq C(d, M) \| \tilde{f}_{r}\|_{2, B^{+}}^{\frac{2}{d+2}}.
\end{equation}
Thus, recalling that by construction $B^{+} \equiv \OOO$ we obtain that by \cref{eq:est1}
\begin{equation*}
\| f \|_{\infty, \OOO} \leq C(d, M) \| \tilde{f}_{r}\|_{2, \OOO}^{\frac{2}{d+2}}.
\end{equation*}

\underline{Part II}:
 Let $x_{0} \in \partial\OO$ be such that $\OO$ is not flat in a neighborhood of $x_{0}$, that is we are in case I. Then, we can find a $C^{1}$ mapping $\Phi$, with inverse given by $\Psi$ such that changing the coordinate system according to the map $\Phi$ we obtain that $\OO^{\prime}:=\Phi(\OO)$ is flat in a neighborhood of $x_{0}$.
 Proceeding similarly as in Part I, we define
 \begin{equation*}
 B^{+}=\overline{B}_{r}(x_{0}) \cap \{ x \in \R^{d}: x_{d} \geq 0\} \subset \overline{\Omega}^{\prime}	
 \end{equation*}
and
\begin{equation*}
B^{-}=\overline{B}_{r}(x_{0}) \cap \{x \in \R^{d}: x_{d} \leq 0\} \subset \R^{d}\backslash \Omega^{\prime}.
\end{equation*}
Thus, if we set $y=\Phi(x)$, we have that $x=\Psi(y)$, and if we define $f^{\prime}(y)=f(\Psi(y))$ then by Parti I we get
\begin{equation*}
\| f^{\prime}\|_{\infty, \Omega^{\prime}} \leq C(d, M) \|f^{\prime} \|_{2, \Omega^{\prime}}^{\frac{2}{2+d}}	
\end{equation*}
which implies, returning to the original coordinates, that
\begin{equation}\label{eq:startinginequality}
	\| f\|_{\infty, \Omega} \leq C(d, M) \|f\|_{2, \Omega}^{\frac{2}{2+d}}
\end{equation}
for a general domain $\OO$ not necessarily flat in a neighborhood of $x_{0} \in \partial\OO$.

Since $\OO$ is compact, there exists a finitely many points $x^{0}_{i} \in \partial\OO$, neighborhood $W_{i}$ is $x^{0}_{i}$ and functions $f^{\prime}_{i}$ defined as before for $i =1, \dots, N$, such that, fixed $W_{0} \subset \OO$, we have $\OO \subset \bigcup_{i=1}^{N} W_{i}$. Furthermore, let $\{\zeta_{i}\}_{i=1, \dots, N}$ be a partition of unit associated with $\{ W_{i}\}_{i =1, \dots, N}$ and define $\bar{f}(x)=\sum_{i=1}^{N}{\zeta_{i}f^{\prime}_{i}(x)}$. Then, by \cref{eq:startinginequality} applied to $\bar{f}$ we get the conclusion.
\end{proof}


\begin{theorem}[Convergence of mild solutions of \cref{eq:lab1}]\label{thm:MR2}
For each $T>1$, let $(u^{T}, m^{\eta}_t)$ be a mild solution of \cref{eq:lab1}. 
Let $(\bar\lambda, \bar u, \bar m)\in\mathcal{S}$. Then, there exists a positive constant $C'$ such that 
\begin{equation}\label{eq:lab38}
\sup_{t \in [0,T]} \Big\|\frac{u^{T}(t, \cdot)}{T} + \bar\lambda\left(1-\frac{t}{T}\right) \Big\|_{\infty, \OOO} \leq \frac{C'}{T^{\frac{1}{d+2}}},
	\end{equation}
	\begin{equation}\label{eq:lab39}
	\frac{1}{T}\int_{0}^{T}{\big\| F(\cdot, m^{\eta}_s)- F(\cdot, \bar m) \big\|_{\infty, \OOO} ds} \leq \frac{C'}{T^{\frac{1}{d+2}}},
	\end{equation}
	where $C'$ depends only on $L$, $F$, $u^{f}$ and $\OO$.
	\end{theorem}

\proof
Let $\bar{v}(x)= \bar{u}(x)-\bar{u}(0)$ and define
	\begin{equation*}
	w(t,x):=\bar{v}(x)-\bar\lambda (T-t), \quad \forall (x,t)\in \OOO\times[0,T].
	\end{equation*}
	Since $(\bar\lambda, \bar u, \bar m)$ is a solution of \cref{eq:MFG1},
one can deduce that $w$ is a constrained viscosity solution of the Cauchy problem
	\begin{align*}\label{cal}
	\begin{split}
	&\begin{cases}
	-\partial_{t} w + H(x, Dw)=F(x, \bar m) \quad \text{in} \quad (0,T)\times \OOO,  \\ w(T,x)=\bar u (x) \quad \quad\quad\quad\quad\quad\,\,\,\,\,\,\,\,\text{in} \quad \OOO.
	\end{cases}
	\end{split}
	\end{align*}
	So, $w(t,x)$ can be represented as the value function of the following minimization problem
	\begin{equation}\label{eq:ww}
	w(t,x)= \inf_{\gamma \in \Gamma_{t,T}(x)} \left\{\int_{t}^{T}{L_{\bar m}\left(\gamma(s), \dot\gamma(s)\right)\ ds} + \bar u(\gamma(T))\right\}, \quad \forall (x,t) \in \OOO \times [0,T].
	\end{equation}
Since $(u^{T}, m^{\eta}_t)$ is a mild solution of \cref{eq:lab1}, in view of \cref{eq:MFGValuefunction} we get that
\begin{equation}\label{eq:www}
	u^T(t,x)= \inf_{\gamma \in \Gamma_{t,T}(x)} \left\{\int_{t}^{T}{L_{m^{\eta}_s}\left(\gamma(s), \dot\gamma(s)\right)\ ds} +  u^f(\gamma(T))\right\}, \quad \forall (x,t) \in \OOO \times [0,T].
	\end{equation}

We prove  inequality \cref{eq:lab39} first. By \Cref{lem:leA} below and H\"{o}lder's inequality, we get
	\begin{align*}
	 \begin{split}
	 & \int_{t}^{T}{\| F(\cdot, m^{\eta}_s) -F(\cdot, \bar m)\| _{\infty, \OOO}\ \frac{ds}{T}} \\
	 \leq  & \ C(\|DF\|_\infty) \int_{t}^{T}{\| F(\cdot, m^{\eta}_s)- F(\cdot, \bar m) \|_{2, \OOO}^{\frac{2}{d+2}}  \frac{ds}{T}} \\
	  \leq & \ \frac{C(\|DF\|_\infty)}{T} \left( \int_{t}^{T}{\| F(\cdot, m^{\eta}_s) -F(\cdot, \bar m) \|_{2,\OOO}^{2}\ ds} \right) ^{\frac{1}{d+2}} \left(\int_{t}^{T}{\bf 1}\ ds  \right) ^{\frac{d+1}{d+2}} .
	  \end{split}
	  \end{align*}
	  Now, by assumption {\bf (F3)} and \Cref{lem:energyestimate} the term
	  $$
	   \left( \int_{t}^{T}{\| F(\cdot, m^{\eta}_s) -F(\cdot, \bar m) \|_{2,\OOO}^{2}\ ds} \right) ^{\frac{1}{d+2}}
	   $$
	  is bounded by a constant depending only on $L$, $F$ and $\OO$, while
	  $$
	  \left(\int_{t}^{T}{ {\bf 1}\ ds } \right) ^{\frac{d+1}{d+2}}\leq T^{\frac{d+1}{d+2}}.
	  $$
	  Inequality \cref{eq:lab39} follows.

Next, we  prove \cref{eq:lab38}. For any given $(x,t)\in\OOO\times[0,T]$, let $\gamma^*:[0,T]\to \OOO$ be a minimizer of problem \cref{eq:ww}. By \cref{eq:ww} and \cref{eq:www}, we have that
\begin{align}\label{eq:lab43}
\begin{split}
	\quad u^{T}(t,x) & -w(t,x)
	 \leq  \ \int_{t}^{T} {L_{m^{\eta}_s}(\gamma^{*}(s), \dot\gamma^{*}(s))\ ds }- \int_{t}^{T} {L_{\bar m}(\gamma^{*}(s), \dot\gamma^{*}(s))\ ds}\\
+\ & u^{f}(\gamma^{*}(T))  - \bar u(\gamma^{*}(T)) = \ u^{f}(\gamma^{*}(T)) 
\\
-\ &  \bar u(\gamma^{*}(T)) + \int_{t}^{T}{ \left( F(\gamma^{*}(s), m^{\eta}_s) - F(\gamma^{*}(s), \bar m) \right)\ ds}.
	\end{split}
	\end{align}

By  \cref{eq:lab43}, we get
\begin{align*}
	 \frac{ u^{T}(t,x) - w(t,x)}{T} \leq &  \underbrace{\bigl|\frac{u^{f}(\gamma^{*}(T))-\bar u (\gamma^{*}(T))}{T}\bigl|}_{A} 
	 \\
	 + & \underbrace{\frac{1}{T}\int_{t}^{T}{\bigl|F(\gamma^{*}(s), m^{\eta}_s) -F(\gamma^{*}(s), \bar m)\bigl| ds}}_{B}.
\end{align*}

Let us first consider term $B$. Note that
\begin{align}\label{eq:5-500}
	& \int_{t}^{T}{\bigl|F(\gamma^{*}(s), m^{\eta}_s)-F(\gamma^{*}(s), \bar m)\bigl|\ \frac{ds}{T}}
	\\
	 \leq &  \int_{t}^{T}{\bigl\|F(\cdot, m^{\eta}_s)-F(\cdot, \bar m)\bigl\|_{\infty,\OOO} \ \frac{ds}{T}}
\leq  \frac{C'}{T^{\frac{1}{d+2}}},
	 \end{align}
where $C'>0$ is a constant depending only on $L$, $F$ and $\OO$.

Since $\bar u$ and $u^f$ are continuous functions on $\OOO$, we can conclude that $A \leq O(\frac{1}{T})$, which together with \cref{eq:5-500} implies that
\[
	  \frac{ u^{T}(t,x) - w(t,x)}{T} \leq\
	  \frac{C''}{T^{\frac{1}{d+2}}}.
	  \]
	  Moreover, for any given $(x,t)\in \OOO\times [0,T]$, let $\xi^{\ast}(\cdot)$ be a minimizer of  problem \cref{eq:www}. In view of \cref{eq:ww} and \cref{eq:www}, we deduce that
\begin{align}\label{eq:lab42}
\begin{split}
	&w(t,x)-u^{T}(t,x)
	\\ \leq & \ \int_{t}^{T} {L_{\bar m}(\xi^{*}(s), \dot\xi^{*}(s))\ ds }+ \bar u(\xi^{*}(T)) - \int_{t}^{T} {L_{m^{\eta}_s}(\xi^{*}(s), \dot\xi^{*}(s))\ ds} - u^{f}(\xi^{*}(T))
	\\=& \ \bar{u}(\xi^{*}(T)) - u^{f}(\xi^{*}(T)) + \int_{t}^{T}{ \left( F(\xi^{*}(s), \bar m) - F(\xi^{*}(s), m^{\eta}_s) \right)\ ds}.
	\end{split}
	\end{align}
So, by almost the same arguments used above, one obtains
\[
	  \frac{w(t,x)- u^{T}(t,x)}{T} \leq\
	  \frac{C''}{T^{\frac{1}{d+2}}},
	  \]
which completes the proof of the theorem.\qed

\appendix

\section{Proof of \Cref{prop:tangentialdiff}}
\label{sec:appendix}

Without any loss of generality, assume there is a curve $\gamma:[-\tau,\tau] \to \OO$ satisfying $\gamma(0)=x$ and
\begin{equation*}
u(\gamma(t_{2}))-u(\gamma(t_{1})) = \int_{t_{1}}^{t_{2}}{L(\gamma(s), \dot\gamma(s))\ ds} + c(t_{2}-t_{1}),
\end{equation*}
for any $[t_{1}, t_{2}] \subset [-\tau,\tau]$, where $\tau>0$ is a constant.
Let $y \in \R^{d}$ be such that $\langle y, \nu(x) \rangle = 0$, where $\nu(x)$ is the outward unit normal to $\partial \Omega$ at $x$. Let $y_{i}\in \R^{d}$ and $\lambda_{i}>0$, $i \in \N$ be such that
\begin{itemize}
\item $y_{i} \to y$ and $\lambda_i \to 0$,\quad as $i \to \infty$;
\item $x+\lambda_{i} y_{i} \in \OOO$,\quad $\forall i\in \N$.
\end{itemize}
By similar arguments in \Cref{prop:diff}, we only need to prove that
\begin{equation}\label{eq:3-70}
\limsup_{i \to \infty} \frac{u(x+\lambda_{i} y_{i})-u(x)}{\lambda_{i}} \leq \langle D_{v}L(x, \dot\gamma(0)), y \rangle \leq \liminf_{i \to \infty} \frac{u(x+\lambda_{i}y_{i})-u(x)}{\lambda_{i}}.
\end{equation}

Now we prove the first inequality in \cref{eq:3-70} and we omit the proof of the second since it follows by a similar argument. For each $i\in\N$, we define a curve $\gamma_{i}: [-\eps, 0] \to \R^{d}$ by
\begin{equation*}
\gamma_{i}(s)=\gamma(s)+\frac{s+\eps}{\eps}\lambda_{i}y_{i},
\end{equation*}
where $0<\eps<\tau$ will be suitably chosen later.
Define the curve $\hat{\gamma}_{i}: [-\eps,0] \to \OOO$ as the projection of $\gamma_{i}$ onto $\OOO$, that is,
\begin{equation*}
\hat{\gamma}_{i}(s)=\gamma(s)+\frac{s+\eps}{\eps}\lambda_{i} y_{i} -d_{\OO}(\gamma_{i}(s))Db(\gamma_{i}(s)).
\end{equation*}
Thus, by \cref{eq:LaxOleinik} and the property of $(u,L,c)$-calibrated curves we have that
\begin{align*}
u(x+\lambda_iy_i)-u(\hat{\gamma}_{i}(-\eps)) \leq & \int_{-\eps}^{0}{L(\hat{\gamma}_{i}(s), \dot{\hat{\gamma}}_{i}(s))\ ds}+c\eps,
\\
u(x)-u(\gamma(-\eps)) = & \int_{-\eps}^{0}{L(\gamma(s), \dot\gamma(s))\ ds}+c\eps.
\end{align*}
Taking the difference of the two expressions, we get
\begin{equation*}
\frac{u(x+\lambda_{i}y_{i})-u(x)}{\lambda_{i}} \leq \frac{1}{\lambda_{i}} \int_{-\eps}^{0}{\Big(L(\hat{\gamma}_{i}(s), \dot{\hat{\gamma}}_{i}(s))-L(\gamma(s), \dot\gamma(s)) \Big)\ ds}.
\end{equation*}
Moreover, by regularity of the data and the Lagrangian $L$ we get
\begin{align*}
\quad\quad & \frac{1}{\lambda_{i}} \int_{-\eps}^{0}{\Big(L(\hat{\gamma}_{i}(s), \dot{\hat{\gamma}}_{i}(s))-L(\gamma(s), \dot\gamma(s)) \Big)\ ds}
\\
= & \int_{-\eps}^{0}{\Big(\frac{s+\eps}{\eps} \big\langle D_{x}L(\gamma(s), \dot\gamma(s)), y_{i} \big\rangle -\frac{1}{\lambda_{i}}d_{\OO}(\gamma_{i}(s)) \big\langle Db(\gamma_{i}(s)), D_{x}L(\gamma(s), \dot\gamma(s))\big\rangle \Big)\ ds}
\\
+& \int_{-\eps}^{0}{\Big(\frac{1}{\eps} \big\langle D_{v}L(\gamma(s), \dot\gamma(s)), y_{i} \big\rangle}
\\
 - & \frac{1}{\lambda_{i}} \big\langle Db(\gamma_{i}(s)), \dot\gamma(s)+\frac{1}{\eps}\lambda_{i}y_{i}\big\rangle {\bf 1}_{\OO^{c}}(\gamma_{i}(s)) \big\langle Db(\gamma_{i}(s)), D_{v}L(\gamma(s), \dot\gamma(s)) \big\rangle  \Big)\ ds
\\
-& \int_{-\eps}^{0}{\frac{1}{\lambda_{i}} d_{\OO}(\gamma_{i}(s)) \big\langle D^{2}b(\gamma_{i}(s))\left(\dot\gamma(s) + \frac{1}{\eps}\lambda_{i}y_{i} \right), D_{v}L(\gamma(s), \dot\gamma(s))\big\rangle\ ds}.
\end{align*}

Let $\eps=\lambda_{i}$. Then,

\noindent ($i$) since $d_{\OO}(\gamma_{i}(s)) \leq \lambda_{i}|y_i|$ and $|D_{x}L(\gamma(s), \dot\gamma(s))| \leq C(1+ \| \dot\gamma\|^2_{\infty})$, we deduce that
	\begin{equation*}
	\frac{1}{\lambda_{i}} d_{\OO}(\gamma_{i}(s)) \big\langle Db(\gamma_{i}(s)), D_{x}L(\gamma(s), \dot\gamma(s)) \big\rangle
	\end{equation*}
	is bounded. Thus,
	\begin{equation*}
	\int_{-\lambda_{i}}^{0}{\frac{1}{\lambda_{i}} d_{\OO}(\gamma_{i}(s)) \big\langle Db(\gamma_{i}(s)), D_{x}L(\gamma(s), \dot\gamma(s)) \big\rangle\ ds} \to 0,
	\end{equation*}
	as $i \to \infty$. Moreover, by similar argument we deduce that
	\begin{equation*}
	\int_{-\eps}^{0}{\frac{1}{\lambda_{i}} d_{\OO}(\gamma_{i}(s)) \big\langle D^{2}b(\gamma_{i}(s))\left(\dot\gamma(s) + \frac{1}{\eps}\lambda_{i}y_{i} \right), D_{v}L(\gamma(s), \dot\gamma(s))\big\rangle\ ds} \to 0,
	\end{equation*}
	as $i \to \infty$.

\noindent ($ii$) Since $b$ and $\gamma$ are smooth functions, we have that
	\begin{align*}
	Db(\gamma_{i}(s))= Db(x)+O(\lambda_{i}),\quad
	\dot\gamma(s)=\dot\gamma(0)+O(\lambda_{i}), \quad i\to\infty.
	\end{align*}
Moreover, $|D_{v}L(\gamma(s), \dot\gamma(s))| \leq C(1+\| \dot\gamma\|_{\infty}).$ Thus,
\begin{align*}
& \int_{-\lambda_{i}}^{0}{\frac{1}{\lambda_{i}} \big\langle Db(\gamma_{i}(s)), \dot\gamma(s)+\frac{1}{\eps}\lambda_{i}y_{i}\big\rangle {\bf 1}_{\OO^{c}}(\gamma_{i}(s)) \big\langle Db(\gamma_{i}(s)), D_{v}L(\gamma(s), \dot\gamma(s)) \big\rangle\ ds}
\\
= & \int_{-\lambda_{i}}^{0}{\frac{1}{\lambda_{i}} \big\langle Db(\gamma_{i}(s)), \dot\gamma(s)\big\rangle {\bf 1}_{\OO^{c}}(\gamma_{i}(s)) \big\langle Db(\gamma_{i}(s)), D_{v}L(\gamma(s), \dot\gamma(s)) \big\rangle\ ds}
\\
&+ \int_{-\lambda_{i}}^{0}{\frac{1}{\lambda_{i}} \big\langle Db(\gamma_{i}(s)), y_{i}\big\rangle {\bf 1}_{\OO^{c}}(\gamma_{i}(s)) \big\langle Db(\gamma_{i}(s)), D_{v}L(\gamma(s), \dot\gamma(s)) \big\rangle\ ds}
\\
= & \int_{-\lambda_{i}}^{0}{\frac{1}{\lambda_{i}} \big\langle Db(x)+O(\lambda_{i}), \dot\gamma(0)+ O(\lambda_{i}) \big\rangle {\bf 1}_{\OO^{c}}(\gamma_{i}(s)) \big\langle Db(\gamma_{i}(s)), D_{v}L(\gamma(s), \dot\gamma(s)) \big\rangle\ ds}
\\
&+ \int_{-\lambda_{i}}^{0}{\frac{1}{\lambda_{i}} \big\langle Db(x)+O(\lambda_{i}), y_{i}\big\rangle {\bf 1}_{\OO^{c}}(\gamma_{i}(s)) \big\langle Db(\gamma_{i}(s)), D_{v}L(\gamma(s), \dot\gamma(s)) \big\rangle\ ds}.
\end{align*}
So, since ${\bf 1}_{\OO^{c}}(\gamma_{i}(s)) \big\langle Db(\gamma_{i}(s)), D_{v}L(\gamma(s), \dot\gamma(s)) \big\rangle$ is bounded we have that the last integrand is bounded independently of $\lambda_{i}$. Furthermore, since $y_{i} \to y$ and $\langle \nu(x), y \rangle=0$ we get
\begin{equation*}
\int_{-\lambda_{i}}^{0}{\frac{1}{\lambda_{i}} \big\langle Db(\gamma_{i}(s)), \dot\gamma(s)+\frac{1}{\eps}\lambda_{i}y_{i}\big\rangle {\bf 1}_{\OO^{c}}(\gamma_{i}(s)) \big\langle Db(\gamma_{i}(s)), D_{v}L(\gamma(s), \dot\gamma(s)) \big\rangle\ ds} \to 0,
\end{equation*}
as $i\to\infty$.

Therefore, we obtain that
\begin{equation*}
\limsup_{i \to \infty} \frac{u(x+\lambda_{i}y_{i})-u(x)}{\lambda_{i}}  \leq \big\langle D_{v}L(x, \dot\gamma(0)), y \big\rangle.
\end{equation*}
\qed

	\noindent {\bf Acknowledgements:}
Piermarco Cannarsa was partly supported by Istituto Nazionate di Alta Matematica (GNAMPA 2019 Research Projects) and by the MIUR Excellence Department Project awarded to the Department of Mathematics, University of Rome Tor Vergata, CUP E83C18000100006. Wei Cheng was partly supported
by Natural Scientific Foundation of China (Grant No. 11871267, 11631006 and 11790272).
Cristian Mendico was partly supported by Istituto Nazionale di Alta Matematica (GNAMPA 2019 Research Projects). Part of this paper was completed while the third author was visiting the Department of Mathematics of the University of Rome Tor Vergata.
Kaizhi Wang was partly supported by National Natural Science Foundation of China (Grant No. 11771283, 11931016).

	\bibliographystyle{unsrt}
\bibliography{references}
	
\end{document}